\documentclass[11pt]{article} 

\usepackage[utf8]{inputenc} 

\usepackage{geometry} 
\usepackage{graphicx} 

\usepackage{booktabs} 
\usepackage{array} 
\usepackage{paralist} 
\usepackage{verbatim} 
\usepackage{subfig} 
\usepackage{amsmath}
\usepackage{amssymb}
\usepackage{amsthm}
\usepackage{mathtools}
\usepackage{bbm}
\usepackage{inputenc}

\usepackage{fancyhdr} 
\usepackage{sectsty}

\usepackage[nottoc,notlof,notlot]{tocbibind} 
\usepackage[titles,subfigure]{tocloft} 

\numberwithin{equation}{section}

\theoremstyle{plain}
\newtheorem{thm}{Theorem}[section]
\newtheorem{lem}[thm]{Lemma}
\newtheorem{prop}[thm]{Proposition}

\theoremstyle{definition}
\newtheorem{defn}[thm]{Definition}
\newtheorem{exmp}[thm]{Example}

\theoremstyle{remark}
\newtheorem{rem}[thm]{Remark}

\def\Acal{\mathcal{A}}

\def\Dcal{\mathcal{D}}
\def\Ecal{\mathcal{E}}
\def\Fcal{\mathcal{F}}

\def\Hcal{\mathcal{H}}

\def\Scal{\mathcal{S}}

\def\Wcal{\mathcal{W}}

\def\Cbb{\mathbb{C}}

\def\Ebb{\mathbb{E}}
\def\Nbb{\mathbb{N}}
\def\Pbb{\mathbb{P}}

\def\Rbb{\mathbb{R}}

\def\Wbb{\mathbb{W}}

\def\1bb{\mathbbm{1}}

\DeclareMathOperator{\HS}{HS}
\let\epsilon\varepsilon
\let\phi\varphi

\title{The damped stochastic wave equation on p.c.f. fractals\footnote{Keywords: Stochastic partial differential equation, wave equation, fractal, Dirichlet form}}
\author{Ben Hambly\footnote{Mathematical Institute, University of Oxford, Woodstock Road, Oxford, OX2 6GG, UK. Email: hambly@maths.ox.ac.uk.}\enspace  and Weiye Yang\footnote{Mathematical Institute, University of Oxford, Woodstock Road, Oxford, OX2 6GG, UK. Email: weiye.yang@maths.ox.ac.uk. ORCiD: 0000-0003-2104-1218.}}
\date{15 March 2018} 

\begin{document}
\maketitle
\begin{abstract}
A p.c.f. fractal with a regular harmonic structure admits an associated Dirichlet form, which is itself associated with a Laplacian. 
This Laplacian enables us to give an analogue of the damped stochastic wave equation on the fractal. We show that a unique 
function-valued solution exists, which has an explicit formulation in terms of the spectral decomposition of the Laplacian. We then use a 
Kolmogorov-type continuity theorem to derive the spatial and temporal H\"{o}lder exponents of the solution. Our results extend 
the analogous results on the stochastic wave equation in one-dimensional Euclidean space. It is known that no function-valued 
solution to the stochastic wave equation can exist in Euclidean dimension two or higher. The fractal spaces that we work with always 
have spectral dimension less than two, and show that this is the right analogue of dimension to express the
``curse of dimensionality'' of the stochastic wave equation. Finally we prove some results on the convergence to equilibrium of the solutions.
\end{abstract}

\section{Introduction}

The aim of this paper is to investigate the properties of some hyperbolic stochastic partial differential equations (SPDEs) on finitely ramified fractals. In 
one dimension \cite{walsh1986} motivated this problem as understanding the behaviour of a guitar string in a sandstorm. That is we have a 
one-dimensional string which is forced by white noise at every point in time and space and are interested in the `music' - the properties of the resulting
waves induced in the string. In the fractal setting we may think of the vibrations of a fractal drum in a sandstorm. For a two-dimensional drum, 
it is known that the solutions to the stochastic wave equation are no longer functions and thus it is of interest to see what happens in the case of
finitely ramified fractals which behave analytically as objects with dimension between one and two. As yet the theory for the behaviour of waves 
propagating through a fractal is much less developed than that for the diffusion of heat and we will not discuss such deterministic waves. 
Instead we consider the regularity properties of the waves starting from rest and arising from forcing by white noise, which are easier to capture, 
as it is the noise and its smoothing via the Laplacian which are crucial to understanding the behaviour of the waves.

The damped stochastic wave equation on $\Rbb^n$, $n \geq 1$ is the SPDE given by
\begin{equation}\label{RnSWE}
\begin{split}
\frac{\partial^2 u}{\partial t^2}(t,x) &= -2\beta \frac{\partial u}{\partial t}(t,x) + \Delta u(t,x) + \xi(t,x),\\
u(0,\cdot) &= \frac{\partial u}{\partial t}(0,\cdot) = 0,
\end{split}
\end{equation}
where $\beta \geq 0$, $\Delta = \Delta_x$ is the Laplacian on $\Rbb^n$ and $\xi$ is a space-time white noise on $[0,\infty) \times \Rbb^n$, where 
we interpret $x \in \Rbb^n$ as space and $t \in [0,\infty)$ as time. The equation \eqref{RnSWE} can equivalently be written as a system of stochastic 
evolution equations in the following way:
\begin{equation*}
\begin{split}
du(t) &= \dot{u}(t)dt,\\
d\dot{u}(t) &= -2\beta \dot{u}(t)dt + \Delta u(t)dt + dW(t),\\
u(0) &= \dot{u}(0) = 0 \in L^2(\Rbb^n),
\end{split}
\end{equation*}
where $W$ is a cylindrical Wiener process on $L^2(\Rbb^n)$, and the solution $u$ and its (formal) derivative $\dot{u}$ are processes taking values in 
some space of functions on $\Rbb^n$. Here we have used instead the differential notation of stochastic calculus, and one should not presume any a 
priori relationship between $u$ and $\dot{u}$. The damped stochastic wave equation (SWE) was introduced in \cite{Cabana1970} in the case $n = 1$, 
and a unique solution was found via a Fourier transform. If $\beta = 0$, there is no damping, and this is the stochastic wave equation. The solution then 
has a neat characterisation given in \cite[Theorem 3.1]{walsh1986} as a rotated modified Brownian sheet in $[0,\infty) \times \Rbb$, and this immediately 
implies that it is jointly H\"{o}lder continuous in space and time for any H\"{o}lder exponent less than $\frac{1}{2}$. These properties, however, do not 
carry over into spatial dimensions $n \geq 2$. Indeed, for $n \geq 2$ a solution to \eqref{RnSWE} still exists, but it is not function-valued. It is necessary 
to expand our space beyond $L^2(\Rbb^n)$ to include certain distributions in order to make sense of the solution. This is related to the fact that 
$n$-dimensional Brownian motion has local times if and only if $n = 1$, see \cite{foondun2011} and further references. There is thus a distinct change 
in the behaviour of the SPDE \eqref{RnSWE} between dimensions $n = 1$ and $n \geq 2$. One of the aims of the present paper is to investigate the 
behaviour of the SPDE in the case that dimension (appropriately interpreted) is in the interval $[1,2)$. When does a function-valued solution exist, and 
if it does, what are its space-time H\"{o}lder exponents? To answer these questions we introduce a class of fractals.

The theory of analysis on fractals started with the construction of a symmetric diffusion on the two-dimensional Sierpinski gasket in \cite{goldstein1987},
\cite{kusuoka1987} and \cite{barlow1988}, which is now known as \textit{Brownian motion on the Sierpinski gasket}. The field has grown quickly since 
then; see\cite{kigami2001} and \cite{barlow1998} for analytic and probabilistic introductions respectively.  In \cite{kigami2001} it is shown that a certain 
class of fractals, known as \textit{post-critically finite self-similar} (or \textit{p.c.f.s.s.}) sets with \textit{regular harmonic structures}, admit operators 
$\Delta$ akin to the Laplacian on $\Rbb^n$. This class includes many well-known fractals such as the $n$-dimensional Sierpinski gasket (for $n \geq 2$)
and the Vicsek fractal, though not the Sierpinski carpet. The operators $\Delta$ generate symmetric diffusions on their respective fractals in the same way 
that the Laplacian on $\Rbb^n$ is the generator of Brownian motion on $\Rbb^n$, and we therefore refer to them also as ``Laplacians'', see 
\cite{barlow1998}. In particular, the existence of a Laplacian $\Delta$ on a given fractal $F$ allows us to formulate PDEs analogous to the heat equation 
and the wave equation on $F$. The heat equation on $F$ has been widely studied, see \cite[Chapter 5]{kigami2001} and many other papers showing results 
such as sub-Gaussian decay of the heat kernel. It is possible in the same way to formulate certain SPDEs on these fractals; for example the stochastic heat 
equation \cite{hambly2016} and, the subject of the present paper,  the damped stochastic wave equation on $F$. The \textit{spectral dimension} $d_s$, 
defined as the exponent for the asymptotic scaling of the eigenvalue counting function of $\Delta$, for any of these fractals satisfies $d_s < 2$, and is the 
correct definition of dimension to use when investigating the analytic properties of the SPDE. Since all of our fractals are compact, we can use spectral 
methods to vastly simplify the problem and find a solution explicitly in terms of the eigenvalues and eigenfunctions of the Laplacian.

Previous work on hyperbolic PDEs and SPDEs on fractals is sparse. The wave equation was first introduced in \cite{kusuoka1987}. 
Since then, there have been two strands of work, either focusing on bounded or on unbounded fractals. In the case of bounded fractals \cite{dalrymple1999}
gave strong evidence that there would be infinite propagation speed for the deterministic wave equation and \cite{hu2002} showed existence and 
uniqueness for a non-linear wave equation. For the unbounded case there is work by \cite{KusuokaZhou1998} and \cite{Strichartz2010} discussing 
the long time behaviour of waves on manifolds with large scale fractal structure and on fractals themselves.  

In \cite{foondun2011} it is mentioned that the stochastic heat equation on certain fractals has a 
so-called ``random-field'' solution as long as the Hausdorff dimension of the fractal is less than 2. The stochastic wave equation is studied elsewhere in 
that paper but an analogous result is not given. In \cite{hambly2016} the stochastic heat equation on p.c.f.s.s. sets with regular harmonic structures is 
shown to have continuous function-valued solutions, as the spectral dimension is less than 2, and its spatial and temporal H\"{o}lder exponents are 
computed; this can be seen to be the direct predecessor of the present paper and is the source of many of the ideas that we use in the following sections.

The structure of the present paper is as follows: In the next subsection we set up the problem, state the precise SPDE to be solved and summarise the 
main results of the paper. In Section \ref{sec:existenceSWE} we make precise the definition of a solution to the damped stochastic wave equation and 
prove the existence of a unique solution $u$ in the form of an $L^2$-valued process. We show that it is a solution in both a ``mild'' sense and a ``weak'' 
sense. Then, in Section \ref{sec:regsol}, we show that this solution is H\"{o}lder continuous in $L^2$ and that the point evaluations $u(t,x)$ are well-defined
random variables. The latter is a necessary condition for us to be able to consider matters of continuity in space and time. In Section \ref{sec:holderst} 
we utilise a Kolmogorov-type continuity theorem for fractals proven in \cite{hambly2016} to deduce the spatial and temporal H\"{o}lder exponents of 
the solution $u$. In Section \ref{sec:equilib} we give results that describe the long-time behaviour of the solutions for any given set of parameters, in 
particular whether or not they eventually settle down into some equilibrium measure.

\subsection{Description of the problem}

We use an identical set-up to \cite{hambly2016}. Let $M \geq 2$ be an integer. Let $(F, (\psi_i)_{i=1}^M)$ be a connected p.c.f.s.s. set (see \cite{kigami2001}) such that $F$ is a compact metric space and the $\psi_i: F \to F$ are injective strict contractions on $F$. Let $I = \{ 1,\ldots,M \}$ and for each $n \geq 0$ let $\Wbb_n = I^n$. Let $\Wbb_* = \bigcup_{n \geq 0} \Wbb_n$ and let $\Wbb = I^\Nbb$. We call the sets $\Wbb_n$, $\Wbb_*$ and $\Wbb$ \textit{word spaces} and we call their elements \textit{words}. Note that $\Wbb_0$ is a singleton containing an element known as the \textit{empty word}. We use the notation $w = w_1w_2w_3\ldots$ with $w_i \in I$ for words $w \in \Wbb_* \cup \Wbb$. For a word $w = w_1, \ldots ,w_n \in \Wbb_*$, let $\psi_w = \psi_{w_1} \circ \cdots \circ \psi_{w_n}$ and let $F_w = \psi_w(F)$. If $w$ is the empty word then $\psi_w$ is the identity on $F$.

If $\Wbb$ is endowed with the standard product topology then there is a canonical continuous surjection $\pi: \Wbb \to F$ given in \cite[Lemma 5.10]{barlow1998}. Let $P \subset \Wbb$ be the post-critical set of $(F, (\psi_i)_{i=1}^M)$, which is finite by assumption. Then let $F^0 = \pi(P)$, and for each $n \geq 1$ let $F^n = \bigcup_{w \in \Wbb_n} \psi_w(F^0)$. Let $F_* = \bigcup_{n = 0}^\infty F^n$. It is easily shown that $(F^n)_{n\geq 0}$ is an increasing sequence of finite subsets and that $F_*$ is dense in $F$.

Let the pair $(A_0,\textbf{r})$ be a regular irreducible harmonic structure on $(F, (\psi_i)_{i=1}^M)$ such that $\textbf{r} = (r_1,\ldots,r_M) \in \Rbb^M$ for some constants $r_i > 0$, $i \in I$ (harmonic structures are defined in \cite[Section 3.1]{kigami2001}). Here \textit{regular} means that $r_i \in (0,1)$ for all $i$. Let $r_{\min} = \min_{i \in I} r_i$ and $r_{\max} = \max_{i \in I} r_i$. If $n \geq 0$, $w = w_1,\ldots w_n \in \Wbb$ then write $r_w := \prod_{i=1}^n r_{w_i}$. Let $d_H > 0$ be the unique real number such that
\begin{equation*}
\sum_{i \in I} r_i^{d_H} = 1.
\end{equation*}
Then let $\mu$ be the Borel regular probability measure on $F$ such that for any $n \geq 0$, if $w \in \Wbb_n$ then $\mu(F_w) = r_w^{d_H}$. In other words, $\mu$ is the self-similar measure on $F$ in the sense of \cite[Section 1.4]{kigami2001} associated with the weights $r_i^{d_H}$ on $I$. Let $(\Ecal,\Dcal)$ be the regular local Dirichlet form on $L^2(F,\mu)$ associated with this harmonic structure, as given by \cite[Theorem 3.4.6]{kigami2001}. This Dirichlet form is associated with a resistance metric $R$ on $F$, defined by
\[ R(x,y) = \left(\inf\{ \Ecal(f,f): f(x)=0, f(y)=1, f\in \Dcal\}\right)^{-1}, \]
which generates the original topology on $F$, by \cite[Theorem 3.3.4]{kigami2001}. Now let $2^{F^0} = \{ b: b \subseteq F^0 \}$ be the power set of $F^0$. Let $\Dcal_{F^0} = \Dcal$, and for proper subsets $b \in 2^{F^0}$ let
\begin{equation*}
\Dcal_b = \{ f\in \Dcal: f|_{F^0 \setminus b}=0 \}.
\end{equation*}
Then similarly to \cite[Corollary 3.4.7]{kigami2001}, $(\Ecal,\Dcal_b)$ is a regular local Dirichlet form on $L^2(F \setminus (F \setminus b),\mu)$. If $b = F^0$ then we may equivalently write $b=N$, and if $b = \emptyset$ then we may equivalently write $b=D$, see \cite{hambly2016}. The letters $N$ and $D$ indicate \textit{Neumann} and \textit{Dirichlet} boundary conditions respectively, and all other values of $b$ indicate a \textit{mixed} boundary condition. Intuitively, $b$ gives the subset of $F^0$ of points that are free to move under the influence of the SPDE, whereas the remaining elements of $F^0$ are fixed at the value 0.

Let $b \in 2^{F^0}$. By \cite[Chapter 4]{barlow1998}, associated with the Dirichlet form $(\Ecal,\Dcal_b)$ on $L^2(F,\mu)$ is a $\mu$-symmetric diffusion $X^b = (X^b_t)_{t \geq 0}$ on $F$ which itself is associated with a $C_0$-semigroup of contractions $S^b = (S^b_t)_{t \geq 0}$ on $L^2(F,\mu)$. Let $\Delta_b$ be the generator of this diffusion. If $b = N$ then $X^N$ has infinite lifetime, by \cite[Lemma 4.10]{barlow1998}. On the other hand, if $b$ is a proper subset of $F^0$ then the process $X^b$ has the law of a version of $X^N$ which is killed at the points $F^0 \setminus b$, by \cite[Section 4.4]{Fukushima2011}. Our notation is identical to that of \cite{hambly2018}.
\begin{exmp}\label{intervalSWE}
(\cite[Example 1.1]{hambly2016}) Let $F = [0,1]$ and take \textit{any} $M \geq 2$. For $1 \leq i \leq M$ let $\psi_i: F \to F$ be the affine map such that $\psi_i(0) = \frac{i-1}{M}$, $\psi_i(1) = \frac{i}{M}$. It follows that $F^0 = \{0,1\}$. Let $r_i = M^{-1}$ for all $i \in I$ and let
\begin{equation*}
A_0 = \left( \begin{array}{cc}
-1 & 1\\
1 & -1
\end{array}
\right).
\end{equation*}
Then all the conditions given above are satisfied. We have $\Dcal = H^1[0,1]$ and $\Ecal(f,g) = \int_0^1 f'g'$. The generators $\Delta_N$ and $\Delta_D$ are respectively the standard Neumann and Dirichlet Laplacians on $[0,1]$. The induced resistance metric $R$ is none other than the standard Euclidean metric.
\end{exmp}
Let $(\Omega, \Fcal, \Pbb)$ be a probability space. The \textit{damped stochastic wave equation} that we consider in the present paper is the SPDE (system) given by
\begin{equation}\label{SWE}
\begin{split}
du(t) &= \dot{u}(t)dt,\\
d\dot{u}(t) &= -2\beta \dot{u}(t)dt + \Delta_bu(t)dt + dW(t),\\
u(0) &= \dot{u}(0) = 0 \in L^2(F,\mu),
\end{split}
\end{equation}
where $\beta \geq 0$ the \textit{damping coefficient} and $b \in 2^{F^0}$ the \textit{boundary conditions} are parameters, and $W = (W(t))_{t \geq 0}$ is a $\Pbb$-cylindrical Wiener process on $L^2(F,\mu)$. That is, $W$ satisfies
\begin{equation*}
\Ebb \left[ \langle f,W(s) \rangle_{L^2(F,\mu)}\langle W(t),g \rangle_{L^2(F,\mu)} \right] = (s \wedge t) \langle f,g \rangle_{L^2(F,\mu)}
\end{equation*}
for all $s,t \in [0,\infty)$ and $f,g \in L^2(F,\mu)$. We would like the solution process $u = (u(t))_{t \geq 0}$ to be $L^2(F,\mu)$-valued, however it is not clear whether or not the same should be required of the first-derivative process $\dot{u} = (\dot{u}(t))_{t \geq 0}$. This will be clarified in the following section.

The main results of the present paper (Theorems \ref{SWEsoln}, \ref{ptregSWE} and \ref{SWEreg}) can be roughly paraphrased as follows:
\begin{thm}
Equip $F$ with its resistance metric $R$. The SPDE \eqref{SWE} has a unique solution which is a stochastic process $u = (u(t,x): (t,x) \in [0,\infty) \times F)$, which is almost surely jointly continuous in $[0,\infty) \times F$. For each $t \in [0,\infty)$, $u(t,\cdot)$ is almost surely essentially $\frac{1}{2}$-H\"{o}lder continuous in $(F,R)$. For each $x \in F$, $u(\cdot,x)$ is almost surely essentially $(1 - \frac{d_s}{2})$-H\"{o}lder continuous in the Euclidean metric, where $d_s \in [1,2)$ is the spectral dimension of $(F,R)$.
\end{thm}
The precise meaning of \textit{essentially} is given in Section \ref{sec:regsol}. We see that the H\"{o}lder exponents given in the above theorem agree with the case ``$F = \Rbb$'' described in the introduction---there we have $d_s = 1$, and the solution is a rotation of a modified Brownian sheet so it has essential H\"{o}lder exponent $\frac{1}{2}$ in every direction. Of course $\Rbb$ is not compact so it doesn't exactly fit into our set-up, but we get a similar result by considering the interval $[0,1]$ instead, see Example \ref{intervalSWE}.
\begin{exmp}[Hata's tree-like set]
See \cite[Figure 1.4]{kigami2001} for a diagram. This p.c.f. fractal takes a parameter $c \in \Cbb$ such that $|c|, |1-c| \in (0,1)$, with $F^0 = \{ c,0,1 \}$, as described in \cite[Example 3.1.6]{kigami2001}. It has a collection of regular harmonic structures given by
\begin{equation*}
A_0 = \left(\begin{array}{ccc}
-h & h & 0\\
h & -(h+1) & 1\\
0 & 1 & -1\\
\end{array}\right)
\end{equation*}
with $\textbf{r} = (h^{-1},1-h^{-2})$ for $h \in (1,\infty)$, and these all fit into our set-up. In the introduction to \cite{walsh1986} the stochastic wave equation on the unit interval is said to describe the motion of a guitar string in a sandstorm (as long as we specify Dirichlet boundary conditions). Likewise, by taking $b = \{ c, 1 \}$ in our tree-like set, we are ``planting'' it at the point 0 so the associated stochastic wave equation approximately describes the motion of a tree in a sandstorm.
\end{exmp}
For more examples see \cite[Example 1.3]{hambly2016}.
\begin{rem}
The resistance metric $R$ is not a particularly intuitive metric on $F$. However, many fractals have a natural embedding in Euclidean space $\Rbb^n$, and subject to mild conditions on $F$ it can be shown that $R$ is equivalent to some positive power of the Euclidean metric, see \cite{hu2006}. An example is the $n$-dimensional Sierpinski gasket described in \cite[Example 1.3]{hambly2016} with $n \geq 2$. In \cite[Section 3]{hu2006} it is shown that there exists a constant $c > 0$ such that
\begin{equation*}
c^{-1}|x-y|^{d_w - d_f} \leq R(x,y) \leq c|x-y|^{d_w - d_f}
\end{equation*}
for all $x,y \in F \subseteq \Rbb^n$, where $d_w = \frac{\log(n+3)}{\log2}$ is the \textit{walk dimension} of the gasket and $d_f = \frac{\log(n+1)}{\log2}$ is its Euclidean Hausdorff dimension. It follows that the above theorem holds (with a different spatial H\"{o}lder exponent) if $R$ is replaced with a Euclidean metric on $F$. We observe that this means that there are function valued solutions to the stochastic wave equation for fractals with arbitrarily large Hausdorff dimension.
\end{rem}

\section{Existence and uniqueness of solution}\label{sec:existenceSWE}

In this section we will make explicit the meaning of a \textit{solution} to the SPDE \eqref{SWE}, and show that such a solution exists and is unique.
\begin{defn}
Henceforth let $\Hcal = L^2(F,\mu)$ and denote its inner product by $\langle \cdot,\cdot \rangle_\mu$. Moreover, for $\lambda > 0$ let $\Dcal^\lambda$ be the space $\Dcal$ equipped with the inner product
\begin{equation*}
\langle \cdot,\cdot \rangle_\lambda := \Ecal(\cdot,\cdot) + \lambda \langle \cdot,\cdot \rangle_\mu.
\end{equation*}
Since $(\Ecal,\Dcal)$ is closed, $\Dcal^\lambda$ is a Hilbert space.
\end{defn}
\begin{rem}
The space $\Dcal$ contains only $\frac{1}{2}$-H\"{o}lder continuous functions since by the definition of the resistance metric we have that
\begin{equation}
|f(x) - f(y)|^2 \leq R(x,y) \Ecal(f,f)
\end{equation}
for all $f \in \Dcal$ and all $x,y \in F$. Therefore, since $\Dcal_b$ is the intersection of the kernels of the continuous linear functionals $\{f \mapsto f(x):\ x \in F^0 \setminus b \}$, it is a closed subset of any $\Dcal^\lambda$ and has finite codimension $|F^0 \setminus b|$.
\end{rem}
\begin{defn}
The unique real $d_H > 0$ such that
\begin{equation*}
\sum_{i \in I} r_i^{d_H} = 1
\end{equation*}
is the \textit{Hausdorff dimension} of $(F,R)$, see \cite[Theorem 1.5.7]{kigami2001}.

The \textit{spectral dimension} of $(F,R)$ is given by
\begin{equation*}
d_s = \frac{2d_H}{d_H + 1},
\end{equation*}
see \cite[Theorem 4.1.5 and Theorem 4.2.1]{kigami2001}. Note by \cite[Remark 2.6]{hambly2016} that $d_H \in [1,\infty)$ and $d_s \in [1,2)$.
\end{defn}
If $A$ is a linear operator on $\Hcal$ then we denote the domain of $A$ by $\Dcal(A)$. If $A$ is bounded, then let $\Vert A \Vert$ be its operator norm. By \cite[Proposition 2.5]{hambly2018}, for each $b \in 2^{F^0}$ there exists an orthonormal basis $(\phi^b_k)_{k=1}^\infty$ of $\Hcal$, where the associated eigenvalues $(\lambda^b_k)_{k=1}^\infty$ are assumed to be in increasing order. In particular any element $f \in \Hcal$ has a series representation
\begin{equation*}
f = \sum_{k=1}^\infty f_k \phi^b_k
\end{equation*}
where $f_k = \langle \phi^b_k,f \rangle_\mu$. Then for any function $\Xi: \Rbb_+ \to \Rbb$, the map $\Xi(-\Delta_b)$ is a well-defined self-adjoint operator from $\Dcal(\Xi(-\Delta_b))$ to $\Hcal$ and has the representation
\begin{equation*}
\Xi(-\Delta_b)f = \sum_{k=1}^\infty f_k \Xi(\lambda^b_k) \phi^b_k,
\end{equation*}
where the domain $\Dcal(\Xi(-\Delta_b)))$ is the subspace of $\Hcal$ of exactly those $f$ for which the above expression is in $\Hcal$. In fact the operator $\Xi(-\Delta_b)$ is densely defined since $\phi^b_k \in \Dcal(\Xi(-\Delta_b)))$ for all $k$. This theory is known as the \textit{functional calculus} for linear operators, see \cite[Theorem VIII.5]{reed1981}.

In particular, if $\alpha \geq 0$ then $(1-\Delta_b)^{\frac{\alpha}{2}}$ is an invertible operator on $\Hcal$, and its inverse $(1-\Delta_b)^{-\frac{\alpha}{2}}$ is a bounded operator on $\Hcal$ which is a bijection from $\Hcal$ to $\Dcal((1-\Delta_b)^{\frac{\alpha}{2}})$.
\begin{defn}
Let $\alpha \geq 0 $ be a real number and $b \in 2^{F^0}$. The bounded operator $(1-\Delta_b)^{-\frac{\alpha}{2}}$ is called a \textit{Bessel potential operator}, see \cite{strichartz2003}, \cite{issoglio2015}. Let $\Hcal^{-\alpha}_b$ be the closure of $\Hcal$ with respect to the inner product given by
\begin{equation*}
(f,g) \mapsto \langle (1-\Delta_b)^{-\frac{\alpha}{2}}f,(1-\Delta_b)^{-\frac{\alpha}{2}}g \rangle_\mu.
\end{equation*}
$\Hcal^{-\alpha}_b$ is called a \textit{Sobolev space}, as in \cite{strichartz2003}.
\end{defn}
\begin{rem}
Recall that $\Dcal((1-\Delta_b)^{\frac{\alpha}{2}})$ is dense in $\Hcal$. It follows that the operator $(1-\Delta_b)^{\frac{\alpha}{2}}: \Dcal((1-\Delta_b)^{\frac{\alpha}{2}}) \to \Hcal$ extends to an isometric isomorphism from $\Hcal$ to $\Hcal^{-\alpha}_b$ characterised by
\begin{equation*}
(1-\Delta_b)^{\frac{\alpha}{2}} \left( \sum_{k=1}^\infty f_k \phi^b_k \right) = \sum_{k=1}^\infty (1+ \lambda^b_k)^{\frac{\alpha}{2}} f_k \phi^b_k.
\end{equation*}
It is easy to see that $\left( (1+ \lambda^b_k)^{\frac{\alpha}{2}} \phi^b_k \right)_{k=1}^\infty$ is a complete orthonormal basis of $\Hcal^{-\alpha}_b$. It follows that $\Hcal$ is dense in $\Hcal^{-\alpha}_b$.
\end{rem}

\subsection{Solution to the SPDE}

Let $\oplus$ denote direct sum of Hilbert spaces. Let $\alpha \geq 0$. The SPDE \eqref{SWE} can be recast as a first-order SPDE on the Hilbert space $\Hcal \oplus \Hcal^{-\alpha}_b$ given by
\begin{equation}\label{SWE2}
\begin{split}
dU(t) &= \Acal_{b,\beta} U(t) dt + d\Wcal(t),\\
U(0) &= 0 \in \Hcal \oplus \Hcal^{-\alpha}_b,
\end{split}
\end{equation}
where
\begin{equation*}
\Acal_{b,\beta} :=
\left(\begin{array}{cc}
0 & 1\\
\Delta_b & -2\beta\\
\end{array}\right)
\end{equation*}
is a densely defined operator on $\Hcal \oplus \Hcal^{-\alpha}_b$ with $\Dcal(\Acal_{b,\beta}) = \Dcal \left(\Delta_b^{(1-\frac{\alpha}{2}) \vee 0} \right) \oplus \Hcal$ and
\begin{equation*}
\Wcal :=
\left(\begin{array}{c}
0\\
W\\
\end{array}\right).
\end{equation*}
There is a precise definition of a solution to evolution equations of the form \eqref{SWE2} which is given in \cite[Chapter 5]{DaPrato1992}, so we can now finally define the notion of a solution to the second-order SPDE \eqref{SWE}. Note that it is still not clear what value of $\alpha$ should be picked.
\begin{defn}
Let $T \in (0,\infty]$. An $\Hcal$-valued predictable process $u = (u(t))_{t=0}^T$ is a \textit{solution} to the SPDE \eqref{SWE} if there exists $\alpha \geq 0$ and an $\Hcal^{-\alpha}_b$-valued predictable process $\dot{u} = (\dot{u}(t))_{t=0}^T$ such that
\begin{equation*}
U :=
\left(\begin{array}{c}
u\\
\dot{u}\\
\end{array}\right)
\end{equation*}
is an $\Hcal \oplus \Hcal^{-\alpha}_b$-valued weak solution to the SPDE \eqref{SWE2} in the sense of \cite[Chapter 5]{DaPrato1992}. If $T = \infty$, then it is a \textit{global} solution.
\end{defn}

Admittedly, the above definition is lacking as it is very abstract and unintuitive. In Theorem \ref{SWEsoln} we prove that solutions to \eqref{SWE} also satisfy a property which is analogous to the concept of weak solution as defined in \cite[Chapter 5]{DaPrato1992}, and is much more instructive.

\begin{defn}
For $\lambda \geq 0$ and $\beta \geq 0$, let $V_\beta(\lambda,\cdot): [0,\infty) \to \Rbb$ be the unique solution to the second-order ordinary differential equation
\begin{equation}
\begin{split}
\frac{d^2v}{dt^2} &+ 2\beta \frac{dv}{dt} + \lambda v = 0,\\
v(0) &= 0,\ \frac{dv}{dt}(0) = 1.
\end{split}
\end{equation}
Explicitly,
\begin{equation*}
V_\beta(\lambda,t) =
\begin{cases}\begin{array}{lr}
(\beta^2 - \lambda)^{-\frac{1}{2}}e^{-\beta t} \sinh \left((\beta^2 - \lambda)^{\frac{1}{2}}t \right) & \lambda < \beta^2,\\
t e^{-\beta t} & \lambda = \beta^2,\\
(\lambda - \beta^2)^{-\frac{1}{2}}e^{-\beta t} \sin \left((\lambda - \beta^2)^{\frac{1}{2}}t \right) & \lambda > \beta^2.\\
\end{array}
\end{cases}
\end{equation*}
For fixed $\lambda$ and $\beta$, this function is evidently smooth in $[0,\infty)$. Let $\dot{V}_\beta(\lambda,\cdot) = \frac{d V_\beta}{d t}(\lambda,\cdot)$.
\end{defn}

\begin{rem}
The different forms of $V_\beta$ correspond respectively to the motion of overdamped, critically damped and underdamped oscillators.
\end{rem}
\begin{lem}\label{lem:Scalconstr}
Let $\alpha = 1$. Then for each $\beta \geq 0$ and $b \in 2^{F^0}$, the operator $\Acal_{b,\beta}$ generates a quasicontraction semigroup $\Scal^{b,\beta} = (\Scal^{b,\beta}_t)_{t \geq 0}$ on $\Hcal \oplus \Hcal^{-1}_b$ such that $\Vert \Scal^{b,\beta}_t \Vert \leq e^{\frac{t}{2}}$ for all $t \geq 0$. Moreover, the right column of $\Scal^{b,\beta}_t$ is given by
\begin{equation*}
\Scal^{b,\beta}_t =
\left(\begin{array}{cc}
\cdot & V_\beta(-\Delta_b,t)\\
\cdot & \dot{V}_\beta(-\Delta_b,t)
\end{array}\right).
\end{equation*}
\end{lem}
\begin{proof}
Recall that
\begin{equation*}
\Acal_{b,\beta} =
\left(\begin{array}{cc}
0 & 1\\
\Delta_b & -2\beta\\
\end{array}\right).
\end{equation*}
If $f \in \Dcal(\Delta_b^\frac{1}{2})$, $g \in \Hcal$ then
\begin{equation*}
\begin{split}
&\left\langle \Acal_{b,\beta}
\left(\begin{array}{c}
f\\
g
\end{array}\right)
,
\left(\begin{array}{c}
f\\
g
\end{array}\right)
\right\rangle_{\Hcal \oplus \Hcal^{-1}_b}\\
&= \langle g,f \rangle_\mu + \langle \Delta_b(1-\Delta_b)^{-\frac{1}{2}} f, (1-\Delta_b)^{-\frac{1}{2}}g \rangle_\mu -2\beta \Vert (1-\Delta_b)^{-\frac{1}{2}}g \Vert_\mu^2\\
&= \langle (1-\Delta_b)(1-\Delta_b)^{-1} f,g \rangle_\mu + \langle \Delta_b(1-\Delta_b)^{-1} f, g \rangle_\mu -2\beta \Vert (1-\Delta_b)^{-\frac{1}{2}}g \Vert_\mu^2\\
&= \langle f,(1-\Delta_b)^{-1} g \rangle_\mu -2\beta \Vert (1-\Delta_b)^{-\frac{1}{2}}g \Vert_\mu^2\\
&\leq \frac{1}{2}\Vert f \Vert_\mu^2 + \frac{1}{2} \Vert (1-\Delta_b)^{-1} g \Vert_\mu^2 -2\beta \Vert (1-\Delta_b)^{-\frac{1}{2}}g \Vert_\mu^2
\end{split}
\end{equation*}
where in the last line we have used the Cauchy-Schwarz inequality. Now $\Vert (1-\Delta_b)^{-\frac{1}{2}} \Vert \leq 1$ by the functional calculus. It follows that
\begin{equation*}
\begin{split}
&\left\langle \left(\Acal_{b,\beta} - \frac{1}{2}\right)
\left(\begin{array}{c}
f\\
g
\end{array}\right)
,
\left(\begin{array}{c}
f\\
g
\end{array}\right)
\right\rangle_{\Hcal \oplus \Hcal^{-1}_b}\\
&\leq -\frac{1}{2}\left( \Vert (1-\Delta_b)^{-\frac{1}{2}} g \Vert_\mu^2 - \Vert (1-\Delta_b)^{-1} g \Vert_\mu^2 \right) -2\beta \Vert (1-\Delta_b)^{-\frac{1}{2}}g \Vert_\mu^2\\
&\leq 0,
\end{split}
\end{equation*}
which implies that the operator $\Acal_{b,\beta} - \frac{1}{2}$ is dissipative. Moreover, it can be easily checked that the operator $\lambda - \Acal_{b,\beta}$ is invertible for any $\lambda > 0$ with bounded inverse
\begin{equation*}
\left(\lambda - \Acal_{b,\beta} \right)^{-1}  =
\left(\begin{array}{cc}
2\beta + \lambda & 1\\
\Delta_b & \lambda\\
\end{array}\right) (\lambda(\lambda + 2\beta) - \Delta_b)^{-1}.
\end{equation*}
It follows by the Lumer--Phillips theorem for reflexive Banach spaces \cite[Corollary II.3.20]{Engel2001} that $\Acal_{b,\beta} - \frac{1}{2}$ generates a contraction semigroup on $\Hcal \oplus \Hcal^{-1}_b$. It follows that $\Acal_{b,\beta}$ generates a quasicontraction semigroup $\Scal^{b,\beta} = (\Scal^{b,\beta}_t)_{t \geq 0}$ on $\Hcal \oplus \Hcal^{-1}_b$ such that $\Vert \Scal^{b,\beta}_t \Vert \leq e^{\frac{t}{2}}$ for all $t \geq 0$.

To construct the semigroup $\Scal$, we first observe that $\Hcal \oplus \Hcal^{-1}_b$ has a complete orthonormal basis given by
\begin{equation*}
\left\{
\left(\begin{array}{c}
\phi^b_k\\
0
\end{array}\right): k \in \Nbb
\right\}
\cup
\left\{
\left(\begin{array}{c}
0\\
(1+\lambda^b_k)^\frac{1}{2}\phi^b_k
\end{array}\right): k \in \Nbb
\right\},
\end{equation*}
and that all of the elements of this basis are in $\Dcal(\Acal_{b,\beta})$. By a density argument, it suffices to compute how $\Acal_{b,\beta}$ affects the elements of this basis. For $k \geq 1$ we see that
\begin{equation*}
\begin{split}
\Acal_{b,\beta} \left(\begin{array}{c}
\phi^b_k\\
0
\end{array}\right) &=
\left(\begin{array}{cc}
0 & 1\\
-\lambda^b_k & -2\beta\\
\end{array}\right) \left(\begin{array}{c}
\phi^b_k\\
0
\end{array}\right),\\
\Acal_{b,\beta} \left(\begin{array}{c}
0\\
(1+\lambda^b_k)^\frac{1}{2}\phi^b_k
\end{array}\right) &=
\left(\begin{array}{cc}
0 & 1\\
-\lambda^b_k & -2\beta\\
\end{array}\right) \left(\begin{array}{c}
0\\
(1+\lambda^b_k)^\frac{1}{2}\phi^b_k
\end{array}\right).
\end{split}
\end{equation*}
Therefore to compute the semigroup $\Scal^{b,\beta}$ it will suffice to take a simple matrix exponential. We see that
\begin{equation*}
\exp\left[ \left(\begin{array}{cc}
0 & 1\\
-\lambda^b_k & -2\beta\\
\end{array}\right) t \right]
= \left(\begin{array}{cc}
\cdot & V_\beta(\lambda^b_k,t)\\
\cdot & \dot{V}_\beta(\lambda^b_k,t)\\
\end{array}\right),
\end{equation*}
where the left column of the matrix is not computed as it is not important. It follows that the semigroup generated by $\Acal_{b,\beta}$ takes the form
\begin{equation*}
\Scal^{b,\beta}_t =
\left(\begin{array}{cc}
\cdot & V_\beta(-\Delta_b,t)\\
\cdot & \dot{V}_\beta(-\Delta_b,t)
\end{array}\right).
\end{equation*}
\end{proof}
\begin{prop}\label{SWE2soln}
Let $\alpha = 1$. Then for each $\beta \geq 0$ and $b \in 2^{F^0}$ there is a unique global $\Hcal \oplus \Hcal^{-1}_b$-valued weak solution $U$ to the SPDE \eqref{SWE2} given by
\begin{equation*}
U(t) =
\left(\begin{array}{c}
\int_0^t V_\beta(-\Delta_b,t-s)dW(s)\\
\int_0^t \dot{V}_\beta(-\Delta_b,t-s)dW(s)
\end{array}\right).
\end{equation*}
In particular, it is a centred Gaussian process and has an $\Hcal \oplus \Hcal^{-1}_b$-continuous version.
\end{prop}
\begin{proof}
Following \cite[Section 5.1.2]{DaPrato1992}, we define the stochastic convolution
\begin{equation*}
W^b_\beta(t) := \int_0^t \Scal^{b,\beta}_{t-s} d\Wcal(t) = \int_0^t \Scal^{b,\beta}_{t-s} \iota_2 dW(t)
\end{equation*}
for $t \geq 0$, where $\iota_2: \Hcal \to \Hcal \oplus \Hcal^{-1}_b$ is the (bounded linear) map $f \mapsto \left(\begin{array}{c}
0\\
f
\end{array}\right)$. For $a \in [0,1)$ we wish to show that
\begin{equation*}
\int_0^T t^{-a} \left\Vert \Scal^{b,\beta}_t \iota_2 \right\Vert^2_{\HS(\Hcal \to \Hcal \oplus \Hcal^{-1}_b)} dt < \infty
\end{equation*}
for all $T > 0$, where $\Vert \cdot \Vert_{\HS(\Hcal \to \Hcal \oplus \Hcal^{-1}_b)}$ denotes the Hilbert-Schmidt norm of operators from $\Hcal$ to $\Hcal \oplus \Hcal^{-1}_b$. We have that
\begin{equation*}
\begin{split}
\int_0^T &t^{-a} \left\Vert \Scal^{b,\beta}_t \iota_2 \right\Vert^2_{\HS(\Hcal \to \Hcal \oplus \Hcal^{-1}_b)} dt = \int_0^T t^{-a} \sum_{k=1}^\infty \left\Vert \Scal^{b,\beta}_t \iota_2 \phi^b_k \right\Vert^2_{\Hcal \oplus \Hcal^{-1}_b} dt\\
&= \sum_{k=1}^\infty \int_0^T t^{-a}
\left\Vert \left(\begin{array}{c}
V_\beta(-\Delta_b,t)\phi^b_k\\
\dot{V}_\beta(-\Delta_b,t)\phi^b_k\\
\end{array}\right) \right\Vert^2_{\Hcal \oplus \Hcal^{-1}_b}
dt\\
&= \sum_{k=1}^\infty \int_0^T t^{-a} V_\beta(\lambda^b_k,t)^2 dt + \sum_{k=1}^\infty \int_0^T t^{-a} (1+ \lambda^b_k)^{-1}\dot{V}_\beta(\lambda^b_k,t)^2 dt,
\end{split}
\end{equation*}
and we treat the above two sums separately.

Now $t \mapsto t^{-a} V_\beta(\lambda^b_k,t)^2$ is always integrable in $[0,T]$ so the only thing that can go wrong is the sum. Since there are only finitely many $k$ such that $\lambda^b_k \leq \beta^2$, it suffices to consider the case $\lambda^b_k > \beta^2$. In this case we have that
\begin{equation*}
\begin{split}
\int_0^T t^{-a} V_\beta(\lambda^b_k,t)^2 dt &= (\lambda^b_k - \beta^2)^{-1} \int_0^T t^{-a} e^{-2\beta t} \sin^2 \left((\lambda^b_k - \beta^2)^{\frac{1}{2}}t \right) dt\\
&\leq (\lambda^b_k - \beta^2)^{-1} (1-a)^{-1} T^{1-a}.
\end{split}
\end{equation*}
It follows that
\begin{equation*}
\begin{split}
\sum_{k=1}^\infty \int_0^T t^{-a} V_\beta(\lambda^b_k,t)^2 dt &\leq \sum_{k: \lambda^b_k \leq \beta^2} \int_0^T t^{-a} V_\beta(\lambda^b_k,t)^2 dt + \frac{T^{1-a}}{1-a} \sum_{k: \lambda^b_k > \beta^2} (\lambda^b_k - \beta^2)^{-1}
\end{split}
\end{equation*}
which is finite by \cite[Proposition 2.5]{hambly2018}. We use a similar method for the $\dot{V}_\beta$ sum. Taking $a=0$, it thus follows from \cite[Theorem 5.4]{DaPrato1992} that the SPDE \eqref{SWE2} has a unique global solution $U = (U(t))_{t=0}^\infty$ in $\Hcal \oplus \Hcal^{-1}_b$ given by
\begin{equation*}
U(t) = W^b_\beta(t) = 
\left(\begin{array}{c}
\int_0^t V_\beta(-\Delta_b,t-s)dW(s)\\
\int_0^t \dot{V}_\beta(-\Delta_b,t-s)dW(s)
\end{array}\right).
\end{equation*}
It is a Gaussian process in $\Hcal \oplus \Hcal^{-1}_b$ by \cite[Theorem 5.2]{DaPrato1992}. As a stochastic integral of a cylindrical Wiener process, it is centred. Moreover, taking $a \in (0,1)$ we see that this $U$ has an $\Hcal \oplus \Hcal^{-1}_b$-continuous version by \cite[Theorem 5.11]{DaPrato1992}.
\end{proof}
\begin{thm}[Solution to \eqref{SWE}]\label{SWEsoln}
There exists a unique global solution $u$ to the SPDE \eqref{SWE}. It is a centred Gaussian process on $\Hcal$ given by
\begin{equation*}
u(t) = \int_0^t V_\beta(-\Delta_b,t-s)dW(s).
\end{equation*}
Moreover, $u$ is the unique $\Hcal$-valued process which satisfies the following ``weak solution'' property: For all $h \in \Dcal(\Delta_b)$, the function $t \mapsto \langle u(t), h \rangle_\mu$ satisfies $\langle u(0), h \rangle_\mu = 0$, is continuous in $[0,\infty)$, and is continuously differentiable in $(0,\infty)$ with
\begin{equation*}
\frac{d}{dt}\langle u(t), h \rangle_\mu = \int_0^t \langle u(s), \Delta_b h \rangle_\mu ds - 2\beta \langle u(t), h \rangle_\mu + \int_0^t \langle h, dW(s) \rangle_\mu.
\end{equation*}
\end{thm}
\begin{proof}
Existence is given directly by Proposition \ref{SWE2soln}, and yields the required centred Gaussian process $u$ as a solution which is continuous in $\Hcal$, with its associated $\dot{u}$ continuous in $\Hcal^{-1}_b$. Now note that our construction of $\Scal^{b,\beta}$ in Lemma \ref{lem:Scalconstr} was independent of the value of $\alpha$. That is, for any $\alpha \geq 0$ such that $\Acal_{b,\beta}$ generates a $C_0$-semigroup on $\Hcal \oplus \Hcal^{-\alpha}_b$, that semigroup must be $\Scal^{b,\beta}$. This means that the process $U$ constructed in Proposition \ref{SWE2soln} is independent of $\alpha$ and thus ensures uniqueness of $u$.

It can be checked directly that the adjoint of the operator $\Acal_{b,\beta}$ is given by
\begin{equation*}
\Acal_{b,\beta}^* =
\left(\begin{array}{cc}
0 & (1-\Delta_b)^{-1} \Delta_b\\
1-\Delta_b & -2\beta\\
\end{array}\right),
\end{equation*}
with domain $\Dcal(\Acal_{b,\beta}^*) = \Dcal(\Delta_b^\frac{1}{2}) \oplus \Hcal = \Dcal(\Acal_{b,\beta})$. By the definition of weak solution in \cite[Chapter 5]{DaPrato1992} for \eqref{SWE2} we see that for all $f \in \Dcal(\Delta_b^\frac{1}{2})$ and $g \in \Hcal$ and $t \in [0,\infty)$,
\begin{equation}\label{eqn:weakSWE}
\begin{split}
&\langle u(t), f \rangle_\mu + \langle \dot{u}(t), g \rangle_{\Hcal^{-1}_b}\\ 
&= \int_0^t\left( \langle u(s), (1-\Delta_b)^{-1} \Delta_b g \rangle_\mu + \langle \dot{u}(s), (1-\Delta_b)f - 2\beta g \rangle_{\Hcal^{-1}_b} \right) ds + \int_0^t \langle g, dW(s) \rangle_{\Hcal^{-1}_b}.
\end{split}
\end{equation}
Take $g=0$ and $f \in \Dcal(\Delta_b^\frac{1}{2})$ in \eqref{eqn:weakSWE}. Then by the fact that $\dot{u}$ is continuous in $\Hcal^{-1}_b$ and the fundamental theorem of calculus, the function $t \mapsto \langle u(t), f \rangle_\mu$ is continuously differentiable in $(0,\infty)$ with
\begin{equation*}
\frac{d}{dt}\langle u, f \rangle_\mu = \langle \dot{u}, (1-\Delta_b)f \rangle_{\Hcal^{-1}_b}.
\end{equation*}
Note in particular that the right-hand side of the above equation is equal to $\langle \dot{u},f \rangle_\mu$ if $\dot{u} \in \Hcal$. Now in \eqref{eqn:weakSWE} we take $f=0$ and let $g = (1-\Delta_b)h$ for some $h \in \Dcal(\Delta_b)$, which gives
\begin{equation*}
\begin{split}
&\langle \dot{u}(t), (1-\Delta_b)h \rangle_{\Hcal^{-1}_b}\\ 
&= \int_0^t\left( \langle u(s), \Delta_b h \rangle_\mu - 2\beta \langle \dot{u}(s), (1-\Delta_b)h \rangle_{\Hcal^{-1}_b} \right) ds + \int_0^t \langle (1-\Delta_b)h, dW(s) \rangle_{\Hcal^{-1}_b},
\end{split}
\end{equation*}
which is equivalent to
\begin{equation*}
\frac{d}{dt}\langle u(t), h \rangle_\mu = \int_0^t \langle u(s), \Delta_b h \rangle_\mu ds - 2\beta \langle u(t), h \rangle_\mu + \int_0^t \langle h, dW(s) \rangle_\mu.
\end{equation*}
Thus $u$ satisfies the required ``weak'' property. It remains to prove that $u$ uniquely satisfies this property among all $\Hcal$-valued processes. In order to do this let $\bar{u}$ be a process also satisfying the property and let $v = u-\bar{u}$. Let $v_k(t) = \langle v(t),\phi^b_k \rangle_\mu$ for $k \geq 1$, $t \in [0,\infty)$. Then $v_k$ can be seen to satisfy the ordinary differential equation
\begin{equation*}
\begin{split}
\frac{d^2v_k}{dt^2} &= -\lambda^b_k v_k - 2\beta \frac{dv_k}{dt},\\
v_k(0) &= \frac{dv_k}{dt}(0) = 0.
\end{split}
\end{equation*}
The unique solution to this ODE is $v_k = 0$ for every $k$, which implies $u = \bar{u}$.
\end{proof}
Now that we have our solution $u$ to \eqref{SWE} given by Theorem \ref{SWEsoln}, we show that it has a nice eigenfunction decomposition. Let $u_k = \langle \phi^b_k u \rangle_\mu$ for $k \geq 1$. We see that
\begin{equation*}
u_k(t) = \int_0^t V_\beta(\lambda^b_k,t-s) \langle \phi^b_k, dW(s) \rangle_\mu,
\end{equation*}
and it can be easily shown that $(\langle \phi^b_k, W \rangle_\mu)_{k=1}^\infty$ is a sequence of independent standard real Brownian motions.
\begin{defn}[Series representation of solution]
Let $\beta \geq 0$ and $b \in 2^{F^0}$. For $k \geq 0$ let $Y^{b,\beta}_k = (Y^{b,\beta}_k(t))_{t \geq 0}$ be the centred real-valued Gaussian process given by
\begin{equation*}
Y^{b,\beta}_k(t) = \int_0^t V_\beta(\lambda^b_k,t-s) \langle \phi^b_k, dW(s) \rangle_\mu.
\end{equation*}
The family $(Y^{b,\beta}_k)_{k=1}^\infty$ is clearly independent, and if $u$ is the solution to \eqref{SWE} for the given values of $\beta$ and $b$, then
\begin{equation}\label{seriesrep}
u(t) = \sum_{k=1}^\infty Y^{b,\beta}_k(t) \phi^b_k.
\end{equation}
\end{defn}
\begin{rem}
By Theorem \ref{SWEsoln}, the real-valued process $Y^{b,\beta}_k$ satisfies the following stochastic integro-differential equation:
\begin{equation*}
\begin{split}
y'(t) &= - 2\beta y(t) - \lambda^b_k \int_0^t y(s)ds + \int_0^t \langle \phi^b_k, dW(s) \rangle_\mu,\\
y(0) &= 0,
\end{split}
\end{equation*}
and it is easily shown to be the unique solution.
\end{rem}
\begin{rem}[Non-zero initial conditions]
For a moment we consider the SPDE
\begin{equation}\label{nonzeroSWE}
\begin{split}
du(t) &= \dot{u}(t)dt,\\
d\dot{u}(t) &= -2\beta \dot{u}(t)dt + \Delta_bu(t)dt + dW(t),\\
u(0) &= f,\ \dot{u}(0) = g.
\end{split}
\end{equation}
This is simply the SPDE \eqref{SWE} with possibly non-zero initial conditions. We can characterise the solutions of this SPDE using the deterministic damped wave equation
\begin{equation}\label{wavePDE}
\begin{split}
du(t) &= \dot{u}(t)dt,\\
d\dot{u}(t) &= -2\beta \dot{u}(t)dt + \Delta_bu(t)dt,\\
u(0) &= f,\ \dot{u}(0) = g,
\end{split}
\end{equation}
which is studied in \cite{dalrymple1999} and \cite{hu2002} in the case $\beta = 0$. Let $u$ be the unique solution to \eqref{SWE} given in Theorem \ref{SWEsoln}. Then it is clear that a process $\tilde{u}$ solves \eqref{nonzeroSWE} if and only if $\tilde{u} - u$ solves \eqref{wavePDE}. Thus understanding the stochastic wave equation with general initial conditions on a fractal is equivalent to understanding the deterministic wave equation on that fractal.
\end{rem}

\section{Regularity of solution}\label{sec:regsol}

\subsection{$L^2$-H\"{o}lder continuity}
The first regularity property of the solution $u = (u(t))_{t=0}^\infty$ to \eqref{SWE} that we will consider is H\"{o}lder continuity in $\Hcal$, when $u$ is interpreted as a function $u:\Omega \times [0,\infty) \to \Hcal$.
\begin{prop}\label{l2estimSWE}
Let $u: \Omega \times [0,\infty) \to \Hcal$ be the solution to the SPDE \eqref{SWE}. For every $T > 0$ there exists a constant $C > 0$ such that
\begin{equation*}
\Ebb\left[ \left\Vert u(s) - u(s+t) \right\Vert^2_\mu\right] \leq Ct^{2 - d_s}
\end{equation*}
for all $s,t \geq 0$ such that $s,s+t \in [0,T]$.
\end{prop}
\begin{proof}
By It\={o}'s isometry for Hilbert spaces,
\begin{equation*}
\begin{split}
\Ebb &\left[ \left\Vert u(s) - u(s+t) \right\Vert^2_\mu\right]\\
&= \Ebb\left[ \left\Vert \int_0^{s+t} \left( V_\beta(-\Delta_b,s+t-t') - V_\beta(-\Delta_b,s-t') \1bb_{\{t' \leq s\}} \right)dW(t') \right\Vert^2_\mu\right]\\
&= \int_0^{s+t} \left\Vert V_\beta(-\Delta_b,s+t-t') - V_\beta(-\Delta_b,s-t') \1bb_{\{t' \leq s\}} \right\Vert^2_{\HS(\Hcal)} dt',
\end{split}
\end{equation*}
where $\Vert \cdot \Vert_{\HS(\Hcal)}$ denotes the Hilbert-Schmidt norm for operators from $\Hcal$ to itself. It follows that
\begin{equation}\label{l2Holdereqn}
\begin{split}
\Ebb &\left[ \left\Vert u(s) - u(s+t) \right\Vert^2_\mu\right]\\
&= \int_0^s \left\Vert V_\beta(-\Delta_b,t+t') - V_\beta(-\Delta_b,t') \right\Vert^2_{\HS(\Hcal)} dt' + \int_0^t \left\Vert V_\beta(-\Delta_b,t') \right\Vert^2_{\HS(\Hcal)} dt'\\
&= \sum_{k=0}^\infty \int_0^s \left( V_\beta(\lambda^b_k,t+t') - V_\beta(\lambda^b_k,t') \right)^2 dt' + \sum_{k=0}^\infty \int_0^t \left( V_\beta(\lambda^b_k,t') \right)^2 dt'
\end{split}
\end{equation}
and we treat each of the above two sums separately. Notice that by \cite[Proposition 2.5]{hambly2018} there are only finitely many $k$ such that $\lambda^b_k \leq \beta^2$.

We consider the first sum of \eqref{l2Holdereqn}, and we first look at the case $\lambda^b_k > \beta^2$. Then using standard facts about the Lipschitz coefficients of the functions $\exp$ and $\sin$ in $[0,T]$ we see that
\begin{equation*}
\begin{split}
&\int_0^s \left( V_\beta(\lambda^b_k,t+t') - V_\beta(\lambda^b_k,t') \right)^2 dt'\\
&= (\lambda^b_k - \beta^2)^{-1} \int_0^s \left( e^{-\beta (t+t')} \sin \left((\lambda^b_k - \beta^2)^\frac{1}{2}(t+t') \right) - e^{-\beta t'} \sin \left((\lambda^b_k - \beta^2)^\frac{1}{2}t' \right) \right)^2 dt'\\
&\leq (\lambda^b_k - \beta^2)^{-1} \int_0^s \left( (\beta+(\lambda^b_k - \beta^2)^{\frac{1}{2}})t \wedge 2 \right)^2 dt'\\
&\leq 4T\frac{\lambda^b_k t^2 \wedge 1}{\lambda^b_k - \beta^2}.
\end{split}
\end{equation*}
We get a similar result in the case $\lambda^b_k \leq \beta^2$, that is, a term of order $O(t^2)$. In this case the dependence of this term on $k$ is unimportant as there are only finitely many $k$ such that $\lambda^b_k \leq \beta^2$. There therefore exists a constant $C' > 0$ such that
\begin{equation*}
\sum_{k=0}^\infty \int_0^s \left( V_\beta(\lambda^b_k,t+t') - V_\beta(\lambda^b_k,t') \right)^2 dt' \leq C't^2 + 4T\sum_{k: \lambda^b_k > \beta^2} \frac{\lambda^b_k t^2 \wedge 1}{\lambda^b_k - \beta^2}.
\end{equation*}
Using \cite[Proposition 2.5]{hambly2018}, there therefore exists $C'' > 0$ such that
\begin{equation*}
\sum_{k=0}^\infty \int_0^s \left( V_\beta(\lambda^b_k,t+t') - V_\beta(\lambda^b_k,t') \right)^2 dt' \leq C''\left(t^2 + \sum_{k=1}^\infty k^{-\frac{2}{d_s}} \wedge t^2 \right).
\end{equation*}
Then by \cite[Lemma 5.2]{hambly2016}, there exists a $C''' > 0$ such that
\begin{equation*}
\sum_{k=0}^\infty \int_0^s \left( V_\beta(\lambda^b_k,t+t') - V_\beta(\lambda^b_k,t') \right)^2 dt' \leq C''' t^{2 - d_s}.
\end{equation*}

Now for the second sum of \eqref{l2Holdereqn}, again we first look at the case $\lambda^b_k > \beta^2$. Using Lipschitz coefficents and the fact that $V_\beta(\lambda^b_k,0) = 0$ we have that
\begin{equation*}
\begin{split}
\int_0^t \left( V_\beta(\lambda^b_k,t') \right)^2 dr &= (\lambda - \beta^2)^{-1}\int_0^t e^{-2\beta t'} \sin^2 \left((\lambda - \beta^2)^{\frac{1}{2}}t' \right)dt'\\
&\leq (\lambda^b_k - \beta^2)^{-1} \int_0^t \left( (\beta+(\lambda^b_k - \beta^2)^{\frac{1}{2}})t' \wedge 1 \right)^2 dt'\\
&\leq 4(\lambda^b_k - \beta^2)^{-1} \int_0^t \left( \lambda^b_k (t')^2 \wedge 1 \right) dt'\\
&\leq 4t \frac{\lambda^b_k t^2 \wedge 1}{\lambda^b_k - \beta^2}.
\end{split}
\end{equation*}
In the case $\lambda^b_k \leq \beta^2$ we get as usual a similar result, of order $O(t^3)$, and its dependence on $k$ is unimportant as there are only finitely many. Using the same method as for the first sum of \eqref{l2Holdereqn} we see that there exists a $C'''' > 0$ such that
\begin{equation*}
\sum_{k=0}^\infty \int_0^t \left( V_\beta(\lambda^b_k,t') \right)^2 dt' \leq C'''' t^{3 - d_s}.
\end{equation*}
Plugging the estimates into \eqref{l2Holdereqn} finishes the proof.
\end{proof}
\begin{defn}
Let $(M_1,d_1)$ and $(M_2,d_2)$ be metric spaces and let $\delta \in (0,1]$. A function $f: M_1 \to M_2$ is \textit{essentially $\delta$-H\"{o}lder continuous} if for each $\gamma \in (0,\delta)$ there exists $C_\gamma > 0$ such that
\begin{equation*}
d_2(f(x),f(y)) \leq C_\gamma d_1(x,y)^\gamma
\end{equation*}
for all $x,y \in M_1$.
\end{defn}
\begin{thm}[$L^2$-H\"{o}lder continuity]\label{thm:L2verSWE}
Let $u: \Omega \times [0,\infty) \to \Hcal$ be the solution to the SPDE \eqref{SWE}. Then there exists a version $\tilde{u}$ of $u$ such that the following holds: for all $T > 0$, the restriction of $\tilde{u}$ to $\Omega \times [0,T]$ is almost surely essentially $(1 - \frac{d_s}{2})$-H\"{o}lder continuous as a function from $[0,T]$ to $\Hcal$.
\end{thm}
\begin{proof}
Fix $T > 0$. This is a simple application of Kolmogorov's continuity theorem. It is a consequence of Fernique's theorem \cite[Theorem 2.7]{DaPrato1992} that for each $p \in \Nbb$ there exists a constant $K_p > 0$ such that if $Z$ is a Gaussian random variable on some separable Banach space $B$ then
\begin{equation*}
\Ebb\left[\left\Vert Z \right\Vert^{2p}_B\right] \leq K_p\Ebb\left[\left\Vert Z \right\Vert^2_B\right]^p,
\end{equation*}
see also \cite[Proposition 3.14]{Hairer2016}. Since $u$ is a Gaussian process, Proposition \ref{l2estimSWE} gives us that
\begin{equation*}
\Ebb\left[ \left\Vert u(s) - u(t) \right\Vert^{2p}_\mu\right] \leq K_p C^p |s - t|^{p(2 - d_s)}
\end{equation*}
for all $s,t \in [0,T]$, for all $p \in \Nbb$. Then by taking $p$ arbitrarily large and using Kolmogorov's continuity theorem, the result follows. Note that any two continuous versions of $u$ must be indistinguishable, which allows us to extend the construction of $\tilde{u}$ on any given finite time interval $[0,T]$ to the whole interval $[0,\infty)$.
\end{proof}

\subsection{Pointwise regularity}

Let $u$ be the solution to \eqref{SWE} in Theorem \ref{SWEsoln}. Henceforth we assume that $u$ is the $\Hcal$-continuous version constructed in Theorem \ref{thm:L2verSWE}. We currently have $u$ as an $\Hcal$-valued process, so in this section we will show that the ``point evaluations'' $u(t,x)$ for $(t,x) \in [0,\infty) \times F$ can be defined in such a way that they make sense as real-valued random variables. This will allow us to interpret $u$ as a function from $\Omega \times [0,\infty) \times F$ to $\Rbb$, and is necessary for us to be able to talk about continuity of $u$ in space and time.
\begin{defn}
For $\lambda > 0$ and $b \in 2^{F^0}$ let $\rho^b_\lambda: F \times F \to \Rbb$ be the \textit{resolvent density} associated with $\Delta_b$, exactly as in \cite[Section 3.1]{hambly2018}.
\end{defn}
\begin{lem}\label{Vint}
Let $\beta \geq 0$ and $\lambda \geq 0$. If $\alpha > 0$ then
\begin{equation*}
\int_0^\infty e^{-2\alpha t} V_\beta(\lambda,t)^2dt = \frac{1}{4(\alpha + \beta)(\alpha^2 + 2\alpha\beta + \lambda)}
\end{equation*}
\end{lem}
\begin{proof}
Can be computed explicitly using (complex) integration in each of the cases $\lambda < \beta^2$, $\lambda = \beta^2$ and $\lambda > \beta^2$ using the definition of $V_\beta$.
\end{proof}
\begin{lem}\label{resolvestimSWE}
Let $u: [0,\infty) \to \Hcal$ be the solution to the SPDE \eqref{SWE}. If $g \in \Hcal$ and $t \in [0,\infty)$ then
\begin{equation*}
\Ebb \left[ \langle u(t),g \rangle_\mu^2 \right] \leq \frac{e^{2(\sqrt{\beta^2 + 1} - \beta)t}}{4\sqrt{\beta^2 + 1}} \int_F \int_F \rho^b_1(x,y)g(x)g(y)\mu(dx)\mu(dy).
\end{equation*}
\end{lem}
\begin{proof}
Let $g^* \in \Hcal^*$ be the bounded linear functional $f \mapsto \langle f,g \rangle_\mu$. We see by It\={o}'s isometry that
\begin{equation*}
\begin{split}
\Ebb \left[ \langle u(t),g \rangle_\mu^2 \right] &= \Ebb \left[ g^*(u(t))^2 \right]\\
&= \int_0^t \Vert g^* V_\beta(-\Delta_b,s) \Vert_{\HS}^2 ds\\
&= \int_0^t \Vert V_\beta(-\Delta_b,s) g \Vert_\mu^2 ds\\
\end{split}
\end{equation*}
where the last equality is a result of the self-adjointness of the operator $V_\beta(-\Delta_b,s)$. If we let $g_k = \langle \phi^b_k, g \rangle_\mu$ for $k \geq 1$ then for any $\alpha > 0$ we have that
\begin{equation*}
\begin{split}
\Ebb \left[ \langle u(t),g \rangle_\mu^2 \right] &= \sum_{k=1}^\infty g_k^2 \int_0^t V_\beta(\lambda^b_k,s)^2 ds\\
&\leq e^{2\alpha t} \sum_{k=1}^\infty g_k^2 \int_0^\infty e^{-2\alpha s} V_\beta(\lambda^b_k,s)^2 ds\\
&= e^{2\alpha t} \sum_{k=1}^\infty g_k^2 \frac{1}{4(\alpha + \beta)(\alpha^2 + 2\alpha\beta + \lambda^b_k)}\\
&= \frac{e^{2\alpha t}}{4(\alpha + \beta)} \left\langle (\alpha^2 + 2\alpha\beta - \Delta_b)^{-1}g , g \right\rangle_\mu\\
&= \frac{e^{2\alpha t}}{4(\alpha + \beta)} \int_F \int_F \rho^b_{\alpha^2+2\alpha\beta}(x,y)g(x)g(y)\mu(dx)\mu(dy),
\end{split}
\end{equation*}
where we have used Lemma \ref{Vint}. Finally we pick $\alpha = \sqrt{\beta^2 + 1} - \beta$ so that $\alpha^2 + 2\alpha\beta = 1$ and the proof is complete.
\end{proof}

For $x \in F$ and $\epsilon > 0$ let $B(x,\epsilon)$ be the closed $R$-ball in $F$ with centre $x$ and radius $\epsilon$.
\begin{lem}[Neighbourhoods]\label{nhoodestimSWE}
There exists a constant $c_5 > 0$ such that the following holds: If $x \in F$ and $n \geq 0$ then there exists a subset $D_n^0(x) \subseteq F$ such that $\mu(D_n^0(x)) > r_{\min}^{d_H}2^{-d_Hn}$ and
\begin{equation*}
x \in D^0_n(x) \subseteq B(x, c_5 2^{-n}).
\end{equation*}
\end{lem}
\begin{proof}
The $D^0_n(x)$ we need is the $n$-neighbourhood of $x$ and is defined in \cite[Definition 3.10]{hambly2016}. The result $D^0_n(x) 
\subseteq B(x, c_5 2^{-n})$ then follows from \cite[Proposition 3.12]{hambly2016}. The result on $\mu(D_n^0(x))$ is due to the fact 
that by definition, $F_w \subseteq D_n^0(x)$ for some $w \in \Wbb_*$ such that $r_w > r_{\min} 2^{-n}$.
\end{proof}
\begin{defn}
For $x \in F$ and $n \geq 0$, define
\begin{equation*}
f^x_n = \mu(D^0_n(x))^{-1} \1bb_{D^0_n(x)}.
\end{equation*}
Evidently $f^x_n \in \Hcal$, $\Vert f^x_n \Vert_\mu^2 = \mu(D^0_n(x))^{-1} < r_{\min}^{-d_H} 2^{d_Hn}$ by the above Lemma and if $g \in \Hcal$ is continuous then
\begin{equation*}
\lim_{n \to \infty}\langle f^x_n,g \rangle_\mu = g(x),
\end{equation*}
by the above lemma.
\end{defn}
We can now state and prove the main theorem of this section, for a similar result for the stochastic heat equation see \cite[Theorem 4.8]{hambly2016}.

\begin{thm}[Pointwise regularity]\label{ptregSWE}
Let $u: [0,\infty) \to \Hcal$ be the solution to the SPDE \eqref{SWE}. Then for all $(t,x) \in [0,\infty) \times F$ the expression
\begin{equation*}
u(t,x) := \sum_{k=1}^\infty Y^{b,\beta}_k(t) \phi^b_k(x)
\end{equation*}
is a well-defined real-valued centred Gaussian random variable. There exists a constant $c_6 > 0$ such that for all $x \in F$, $t \in [0,\infty)$ and $n \geq 0$ we have that
\begin{equation*}
\Ebb \left[ \left( \langle u(t), f^x_n \rangle_\mu - u(t,x) \right)^2 \right] \leq c_6e^{2(\sqrt{\beta^2 + 1} - \beta)t} 2^{-n}.
\end{equation*}
\end{thm}

\begin{proof}
Note that $\phi^b_k \in \Dcal(\Delta_b)$ for each $k$, so $\phi^b_k$ is continuous and so $\phi^b_k(x)$ is well-defined. By the definition of $u(t,x)$ as a sum of real-valued centred Gaussian random variables we need only prove that it is square-integrable and that the approximation estimate holds. Let $x \in F$. The theorem is trivial for $t = 0$ so let $t \in (0,\infty)$. By Lemma \ref{resolvestimSWE} we have that
\begin{equation*}
\begin{split}
\Ebb &\left[ \langle u(t),f^x_n - f^x_m \rangle_\mu^2 \right]\\
&\leq \frac{e^{2(\sqrt{\beta^2 + 1} - \beta)t}}{4\sqrt{\beta^2 + 1}} \int_F \int_F \rho^b_1(z_1,z_2)(f^x_n(z_1)-f^x_m(z_1))(f^x_n(z_2)-f^x_m(z_2))\mu(dz_1)\mu(dz_2).
\end{split}
\end{equation*}
Then using the definition of $f^x_n$, \cite[Proposition 3.2]{hambly2018} and Lemma \ref{nhoodestimSWE} we have that
\begin{equation}\label{cauchyseqSWE}
\begin{split}
\Ebb \left[ \langle u(t),f^x_n - f^x_m \rangle_\mu^2 \right] &\leq \frac{e^{2(\sqrt{\beta^2 + 1} - \beta)t}}{4\sqrt{\beta^2 + 1}} \left( 8c_52^{-n} + 8c_52^{-m} \right)\\
&= \frac{2c_5 e^{2(\sqrt{\beta^2 + 1} - \beta)t}}{\sqrt{\beta^2 + 1}} \left( 2^{-n} + 2^{-m} \right).
\end{split}
\end{equation}
Writing $u$ in its series representation \eqref{seriesrep} and using the independence of the $Y^{b,\beta}_k$, it follows that
\begin{equation*}
\sum_{k=1}^\infty \Ebb \left[Y^{b,\beta}_k(t)^2 \right] \left( \langle \phi^b_k, f^x_n\rangle_\mu - \langle \phi^b_k,f^x_m \rangle_\mu \right)^2 \leq \frac{2c_5 e^{2(\sqrt{\beta^2 + 1} - \beta)t}}{\sqrt{\beta^2 + 1}} \left( 2^{-n} + 2^{-m} \right).
\end{equation*}
Thus the left-hand side of the above equation tends to zero as $m,n \to \infty$. The solution $u$ is an $\Hcal$-valued Gaussian process so we know that
\begin{equation*}
\sum_{k=1}^\infty \Ebb \left[Y^{b,\beta}_k(t)^2 \right] \langle \phi^b_k, f^x_n\rangle_\mu^2 = \Ebb \left[ \langle u(t),f^x_n \rangle_\mu^2 \right] < \infty
\end{equation*}
for all $x \in F$, $n \geq 0$ and $t \in [0,\infty)$, therefore by the completeness of the sequence space $\ell^2$ there must exist a unique sequence $(y_k)_{k=1}^\infty$ such that $\sum_{k=1}^\infty y_k^2 < \infty$ and
\begin{equation*}
\lim_{n \to \infty} \sum_{k=1}^\infty \left( \Ebb \left[Y^{b,\beta}_k(t)^2 \right]^\frac{1}{2} \langle \phi^b_k, f^x_n\rangle_\mu - y_k \right)^2 = 0.
\end{equation*}
Since $\phi^b_k$ is continuous we have $\lim_{n \to \infty}\langle \phi^b_k, f^x_n\rangle_\mu = \phi^b_k(x)$. Thus by Fatou's lemma we can identify the sequence $(y_k)$; we must have
\begin{equation*}
\sum_{k=1}^\infty \Ebb \left[Y^{b,\beta}_k(t)^2 \right] \phi^b_k(x)^2 < \infty
\end{equation*}
and
\begin{equation*}
\lim_{n \to \infty}\sum_{k=1}^\infty \Ebb \left[Y^{b,\beta}_k(t)^2 \right] \left( \langle \phi^b_k, f^x_n\rangle_\mu - \phi^b_k(x) \right)^2 = 0.
\end{equation*}
Equivalently by \eqref{seriesrep},
\begin{equation*}
\Ebb \left[ u(t,x)^2 \right] < \infty
\end{equation*}
(so we have proven square-integrability) and
\begin{equation*}
\lim_{n \to \infty}\Ebb \left[ \left( \langle u(t), f^x_n \rangle_\mu - u(t,x) \right)^2 \right] = 0.
\end{equation*}
In particular by taking $m \to \infty$ in \eqref{cauchyseqSWE} we have that
\begin{equation*}
\Ebb \left[ \left( \langle u(t), f^x_n \rangle_\mu - u(t,x) \right)^2 \right] \leq \frac{2c_5 e^{2(\sqrt{\beta^2 + 1} - \beta)t}}{\sqrt{\beta^2 + 1}} 2^{-n}.
\end{equation*}
\end{proof}
We can now interpret our solution $u$ as a so-called ``random field'' solution $u: \Omega \times [0,\infty) \times F \to \Rbb$. However, the relationship between the random field solution and the original $\Hcal$-valued solution is still rather unclear. We discuss this in the next section.

\section{Space-time H\"{o}lder continuity}\label{sec:holderst}

Now that we have the interpretation of the solution $u$ to \eqref{SWE} as a function $u: \Omega \times [0,\infty) \times F \to \Rbb$, we can prove results about 
its continuity in time and space. In particular, we show that it has a H\"older continuous version which is also a version of the original $\Hcal$-valued solution 
found in Theorem \ref{SWEsoln}.

\subsection{Spatial estimate}

The spatial continuity of $u$ is the same as for the stochastic heat equation, see \cite[Section 5.1]{hambly2016}.
\begin{prop}\label{spaceestimSWE}
Let $T > 0$. Let $u: \Omega \times [0,T] \times F \to \Rbb$ be (the restriction of) the solution to the SPDE \eqref{SWE}. Then there exists a constant $C_1 > 0$ such that
\begin{equation*}
\Ebb \left[ (u(t,x) - u(t,y))^2 \right] \leq C_1R(x,y)
\end{equation*}
for all $t \in [0,T]$ and all $x,y \in F$.
\end{prop}
\begin{proof}
Recall from Theorem \ref{ptregSWE} that
\begin{equation*}
\lim_{n \to \infty}\Ebb \left[ \left( \left\langle u(t), f^x_n \right\rangle_\mu - u(t,x) \right)^2 \right] = 0,
\end{equation*}
and an analogous result holds for $y$. Thus by Lemma \ref{resolvestimSWE},
\begin{equation*}
\begin{split}
\Ebb &\left[ (u(t,x) - u(t,y))^2 \right] = \lim_{n \to \infty} \Ebb \left[ \left\langle u(t), f^x_n - f^y_n \right\rangle_\mu^2 \right]\\
&\leq \frac{e^{2(\sqrt{\beta^2 + 1} - \beta)t}}{4\sqrt{\beta^2 + 1}} \lim_{n \to \infty}\int_F \int_F \rho^b_1(z_1,z_2)(f^x_n(z_1) - f^y_n(z_1))(f^x_n(z_2) - f^y_n(z_2))\mu(dz_1)\mu(dz_2)\\
&= \frac{e^{2(\sqrt{\beta^2 + 1} - \beta)t}}{4\sqrt{\beta^2 + 1}} \left( \rho^b_1(x,x) - 2\rho^b_1(x,y) + \rho^b_1(y,y) \right),
\end{split}
\end{equation*}
where we have used the continuity of the resolvent density, Lemma \ref{nhoodestimSWE}, and the definition of $f^x_n$ (similarly to the proof of Theorem \ref{ptregSWE}). Hence by \cite[Proposition 3.2]{hambly2018},
\begin{equation*}
\begin{split}
\Ebb \left[ (u(t,x) - u(t,y))^2 \right] &\leq \frac{e^{2(\sqrt{\beta^2 + 1} - \beta)T}}{4\sqrt{\beta^2 + 1}} \left( \rho^b_1(x,x) - \rho^b_1(x,y) + \rho^b_1(y,y) - \rho^b_1(y,x) \right)\\
&\leq \frac{e^{2(\sqrt{\beta^2 + 1} - \beta)T}}{\sqrt{\beta^2 + 1}} R(x,y).
\end{split}
\end{equation*}
\end{proof}

\subsection{Temporal estimate}
\begin{lem}\label{opnormestim}
We have the following estimates on $V_\beta$ and $\dot{V}_\beta$:
\begin{enumerate}
\item Let $\beta \geq 0$ and $t \geq 0$. Then
\begin{equation*}
\sup_{\lambda \geq 0} |V_\beta(\lambda,t)| =
\begin{cases}\begin{array}{lr}

\beta^{-1} e^{-\beta t} \sinh \left( \beta t \right) & \beta > 0,\\
t & \beta = 0.
\end{array}\end{cases}
\end{equation*}
In particular, $\sup_{\lambda \geq 0} |V_\beta(\lambda,t)|$ is $O(t)$ as $t \to 0$.
\item Let $\beta \geq 0$ and $T \geq 0$. Then
\begin{equation*}
\sup_{0 \leq t \leq T} \sup_{\lambda \geq 0} |\dot{V}_\beta(\lambda,t)| \leq e^{\beta T}.
\end{equation*}
\end{enumerate}
\end{lem}
\begin{proof}
It is easy, if somewhat tedious, to prove that $V_\beta$ and $\dot{V}_\beta$ are both continuous in $\lambda$ for fixed $t \geq 0$. Note that
\begin{equation*}
\lim_{x \to 0}\frac{\sin x}{x} = 1 = \lim_{x \to 0}\frac{\sinh x}{x}
\end{equation*}
and
\begin{equation*}
\sup_{x \in \Rbb \setminus \{ 0 \}} \left\vert \frac{\sin x}{x} \right\vert = 1 = \inf_{x \in \Rbb \setminus \{ 0 \}} \left\vert \frac{\sinh x}{x} \right\vert.
\end{equation*}

For (1), assume that $t > 0$ (otherwise the result is trivial). We have that
\begin{equation*}
\begin{split}
\sup_{\lambda > \beta^2} |V_\beta(\lambda,t)| &= te^{-\beta t} \sup_{\lambda > \beta^2} \left\vert\left((\lambda - \beta^2)^{\frac{1}{2}}t \right)^{-1} \sin \left((\lambda - \beta^2)^{\frac{1}{2}}t \right) \right\vert\\
&= te^{-\beta t}\\
&= |V_\beta(\beta^2,t)|,
\end{split}
\end{equation*}
so we need only consider the case $\lambda \leq \beta^2$. If $\beta = 0$ then this directly implies the result. Suppose now that $\beta > 0$. The function $x \mapsto \frac{\sinh x}{x}$ is positive and increasing when $x$ is positive so by continuity we have that
\begin{equation*}
\begin{split}
\sup_{\lambda \geq 0} |V_\beta(\lambda,t)| &= \sup_{\lambda \leq \beta^2} |V_\beta(\lambda,t)|\\
&= te^{-\beta t} \sup_{\lambda \leq \beta^2} \left(\left((\beta^2 - \lambda)^{\frac{1}{2}}t \right)^{-1} \sinh \left((\beta^2 - \lambda)^{\frac{1}{2}}t \right) \right)\\
&= te^{-\beta t} \left( \beta t \right)^{-1} \sinh \left(\beta t \right)\\
&= \beta^{-1} e^{-\beta t} \sinh \left( \beta t \right)
\end{split}
\end{equation*}
which is the required result.

Now for (2), assume that $T > 0$, otherwise the result is trivial. We have
\begin{equation*}
\begin{split}
\sup_{0 \leq t \leq T}& \sup_{\lambda > \beta^2}|\dot{V}_\beta(\lambda,t)|\\
&= \sup_{0 \leq t \leq T} \sup_{\lambda > \beta^2} \left\vert e^{-\beta t} \cos \left((\lambda - \beta^2)^{\frac{1}{2}}t \right) - \beta (\lambda - \beta^2)^{-\frac{1}{2}}e^{-\beta t} \sin \left((\lambda - \beta^2)^{\frac{1}{2}}t \right) \right\vert\\
&\leq 1 + \sup_{0 \leq t \leq T} \left\vert \beta t e^{-\beta t} \right\vert\\
&\leq 1 + \beta T
\end{split}
\end{equation*}
and
\begin{equation*}
\begin{split}
\sup_{0 \leq t \leq T}& \sup_{\lambda < \beta^2}|\dot{V}_\beta(\lambda,t)|\\
&= \sup_{0 \leq t \leq T} \sup_{\lambda < \beta^2} \left\vert e^{-\beta t} \cosh \left((\beta^2 - \lambda)^{\frac{1}{2}}t \right) - \beta (\beta^2 - \lambda)^{-\frac{1}{2}}e^{-\beta t} \sinh \left((\beta^2 - \lambda )^{\frac{1}{2}}t \right) \right\vert\\
&\leq \cosh (\beta T) + \beta T \sup_{0 \leq t \leq T} \sup_{\lambda < \beta^2} \left( \left( (\beta^2 - \lambda)^{\frac{1}{2}} t \right)^{-1} \sinh \left((\beta^2 - \lambda )^{\frac{1}{2}}t \right) \right) \\
&\leq \cosh (\beta T) +  \sinh (\beta T) = e^{\beta T}
\end{split}
\end{equation*}
and $\sup_{0 \leq t \leq T} |\dot{V}_\beta(\beta^2,t)| = \sup_{0 \leq t \leq T} \left\vert e^{-\beta t} - \beta t e^{-\beta t} \right\vert \leq 1 + \beta T$. Finally we note that the inequality $1 + \beta T \leq e^{\beta T}$ holds.
\end{proof}
We can now give the temporal estimate. Here we see the effect of the extra time derivative compared to the stochastic heat 
equation~\cite[Proposition 5.5]{hambly2016}.

\begin{prop}\label{timeestimSWE}
Let $T > 0$. Let $u: \Omega \times [0,T] \times F \to \Rbb$ be (the restriction of) the solution to the SPDE \eqref{SWE}. Then there exists $C_2 > 0$ such that
\begin{equation*}
\Ebb \left[ (u(s,x) - u(s+t,x))^2 \right] \leq C_2 t^{2 - d_s}
\end{equation*}
for all $s,t \geq 0$ such that $s, s+t \leq T$ and all $x \in F$.
\end{prop}
\begin{proof}
Let $c_6' := 8c_6e^{2(\sqrt{\beta^2 + 1} - \beta)T}$, where $c_6$ is from Theorem \ref{ptregSWE}. By Theorem \ref{ptregSWE} we have that if $n \geq 0$ is an integer then
\begin{equation}\label{timebd1}
\Ebb \left[ (u(s,x) - u(s+t,x))^2 \right] \leq 2\Ebb \left[ \langle u(s) - u(s+t), f^x_n \rangle_\mu^2 \right] + c_6' 2^{-n}.
\end{equation}
Then It\={o}'s isometry for Hilbert spaces (see also proof of Lemma \ref{resolvestimSWE}) gives us that
\begin{equation*}
\begin{split}
\Ebb &\left[ \langle u(s) - u(s+t), f^x_n \rangle_\mu^2 \right]\\
&= \Ebb \left[ \left\langle \int_0^{s+t} \left( V_\beta(-\Delta_b,s+t - t') - V_\beta(-\Delta_b,s - t')\1bb_{\{t' \leq s\}} \right) dW(t'), f^x_n \right\rangle_\mu^2 \right]\\
&= \int_0^{s+t} \left\Vert \left( V_\beta(-\Delta_b,s+t - t') - V_\beta(-\Delta_b,s - t')\1bb_{\{t' \leq s\}} \right)f^x_n \right\Vert_\mu^2 dt'\\
&\leq \Vert f^x_n \Vert_\mu^2 \int_0^{s+t} \left\Vert V_\beta(-\Delta_b,s+t - t') - V_\beta(-\Delta_b,s - t')\1bb_{\{t' \leq s\}} \right\Vert^2 dt'.
\end{split}
\end{equation*}
Recall that $\Vert f^x_n \Vert_\mu^2 < r_{\min}^{-d_H} 2^{d_Hn}$. Using the functional calculus we see that
\begin{equation*}
\begin{split}
&\int_0^{s+t} \left\Vert V_\beta(-\Delta_b,s+t - t') - V_\beta(-\Delta_b,s - t')\1bb_{\{t' \leq s\}} \right\Vert^2 dt'\\
&= \int_0^s \left\Vert V_\beta(-\Delta_b,s+t - t') - V_\beta(-\Delta_b,s - t') \right\Vert^2 dt' + \int_s^{s+t} \left\Vert V_\beta(-\Delta_b,s+t - t') \right\Vert^2 dt'\\
&= \int_0^s \left\Vert V_\beta(-\Delta_b,t + t') - V_\beta(-\Delta_b,t') \right\Vert^2 dt' + \int_0^t \left\Vert V_\beta(-\Delta_b,t') \right\Vert^2 dt'\\
&\leq \int_0^s \sup_{\lambda \geq 0} \left(V_\beta(\lambda,t + t') - V_\beta(\lambda,t') \right)^2 dt' + \int_0^t \sup_{\lambda \geq 0} V_\beta(\lambda,t')^2 dt'\\
&\leq t^2 T \sup_{0 \leq t' \leq T} \sup_{\lambda \geq 0} \dot{V}_\beta(\lambda,t')^2 + \int_0^t \sup_{\lambda \geq 0} V_\beta(\lambda,t')^2 dt',
\end{split}
\end{equation*}
where in the last line we have used the mean value theorem. Therefore by using Lemma \ref{opnormestim} there exists $c>0$ such that
\begin{equation*}
\int_0^{s+t} \left\Vert V_\beta(-\Delta_b,s+t - t') - V_\beta(-\Delta_b,s - t')\1bb_{\{t' \leq s\}} \right\Vert^2 dt' \leq ct^2
\end{equation*}
for all $s,t \geq 0$ such that $s,s+t \leq T$. Letting $c' = 2 r_{\min}^{-d_H} c$ and plugging this into \eqref{timebd1} we have that
\begin{equation*}
\Ebb \left[ (u(s,x) - u(s+t,x))^2 \right] \leq c' t^2 2^{d_H n} + c_6' 2^{-n}.
\end{equation*}
for all $s,t \geq 0$ such that $s,s+t \leq T$ and all $x \in F$. In fact, defining
\begin{equation*}
c_6'' :=  c_6' \vee d_Hc' T^2,
\end{equation*}
we have that
\begin{equation}\label{timebd2}
\Ebb \left[ (u(s,x) - u(s+t,x))^2 \right] \leq c' t^2 2^{d_H n} + c_6'' 2^{-n}
\end{equation}
as well. This estimate will turn out to be easier to work with.

We assume now that $t > 0$, and our aim is to choose $n \geq 0$ to minimise the expression on the right of \eqref{timebd2}. Fixing $t \in (0,T]$, define $g: \Rbb \to \Rbb_+$ such that $g(y) = c't^2 2^{d_Hy} + c_6''2^{-y}$. The function $g$ has a unique stationary point which is a global minimum at
\begin{equation*}
y_0 = \frac{1}{(d_H + 1)\log 2} \log \left( \frac{c_6''}{d_Hc't^2} \right).
\end{equation*}
Since $t \leq T$ we have by the definition of $c_6''$ that $y_0 \geq 0$. Since $y_0$ is not necessarily an integer we choose $n = \lceil y_0 \rceil$. Then $g$ is increasing in $[y_0,\infty)$ so we have that
\begin{equation*}
\Ebb \left[ (u(s,x) - u(s+t,x))^2 \right] \leq g(n) \leq g(y_0 + 1).
\end{equation*}
Setting $c_6''' := \frac{c_6''}{d_H c'}$ and evaluating the right-hand side we see that
\begin{equation*}
\begin{split}
\Ebb \left[ (u(s,x) - u(s+t,x))^2 \right] &\leq c' t^2 2^{d_H} \left( \frac{c'''_6}{t^2} \right)^\frac{d_H}{d_H + 1} + c_6' 2^{-1} \left( \frac{c'''_6}{t^2} \right)^\frac{-1}{d_H + 1}\\
&\leq c_6'''' t^\frac{2}{d_H + 1}\\
&= c_6'''' t^{2 - d_s}
\end{split}
\end{equation*}
for all $s \geq 0$, $t > 0$ such that $s,s+t \leq T$ and all $x \in F$, where the constant $c_6'''' > 0$ is independent of $s,t,x$. This inequality obviously also holds in the case $t=0$.
\end{proof}

\subsection{H\"{o}lder continuity}
We are now ready to prove the main result of this paper. 
\begin{defn}
Let $R_\infty$ be the metric on $\Rbb \times F$ given by
\begin{equation*}
R_\infty((s,x),(t,y)) = |s-t| \vee R(x,y).
\end{equation*}
\end{defn}
\begin{thm}[Space-time H\"{o}lder continuity]\label{SWEreg}
Let $u: \Omega \times [0,\infty) \times F \to \Rbb$ be the solution to the SPDE \eqref{SWE}. Let $\delta = 1 - \frac{d_s}{2}$. Then there exists a version $\tilde{u}$ of $u$ which satisfies the following:
\begin{enumerate}
\item For each $T > 0$, $\tilde{u}$ is almost surely essentially $(\frac{1}{2} \wedge \delta )$-H\"{o}lder continuous on $[0,T] \times F$ with respect to $R_\infty$.
\item For each $t \in [0,\infty)$, $\tilde{u}(t,\cdot)$ is almost surely essentially $\frac{1}{2}$-H\"{o}lder continuous on $F$ with respect to $R$.
\item For each $x \in F$, $\tilde{u}(\cdot,x)$ is almost surely essentially $\delta$-H\"{o}lder continuous on $[0,T]$ with respect to the Euclidean metric.
\end{enumerate}
Moreover, the collection of random variables $\tilde{u} = (\tilde{u}(t,x))_{(t,x) \in [0,\infty) \times F}$ is such that $(\tilde{u}(t,\cdot))_{t \in [0,\infty)}$ is an $\Hcal$-valued process, and moreover $(\tilde{u}(t,\cdot))_{t \in [0,\infty)}$ is an $\Hcal$-continuous version of the $\Hcal$-valued solution to \eqref{SWE} found in Theorem \ref{SWEsoln}.
\end{thm}
\begin{proof}
Take $T > 0$ and consider $u_T$, the restriction of $u$ to $[0,T] \times F$. It is a well-known fact that for every $p \in \Nbb$ there exists a constant $C_p' > 0$ such that if $Z$ is any centred real Gaussian random variable then
\begin{equation*}
\Ebb [Z^{2p}] = C_p' \Ebb[Z^2]^p.
\end{equation*}
We know that $u_T$ is a centred Gaussian process on $[0,T] \times F$ by Theorem \ref{ptregSWE}. Propositions \ref{spaceestimSWE} and \ref{timeestimSWE} then give us the estimates
\begin{equation*}
\begin{split}
\Ebb \left[ (u_T(t,x) - u_T(t,y))^{2p} \right] &\leq C_p'C_1^pR(x,y)^p,\\
\Ebb \left[ (u_T(s,x) - u_T(t,x))^{2p} \right] &\leq C_p'C_2^p|s-t|^{p(2 - d_s)}
\end{split}
\end{equation*}
for all $s,t \in [0,T]$ and all $x,y \in F$. The existence of a version $\tilde{u}$ with the required H\"older continuity properties then follows in the same way as in \cite[Theorem 5.6]{hambly2016}. Then using Theorem \ref{ptregSWE} and the series representation of $u$, the rest of the present theorem follows in the same way as in \cite[Theorem 5.7]{hambly2016}.
\end{proof}

\section{Convergence to equilibrium}\label{sec:equilib}

We conclude this paper with a brief discussion of the long-time behaviour of the solution $u$ to the SPDE \eqref{SWE}. We are interested in whether the solution ``settles down'' as $t \to \infty$ to some equilibrium measure. Intuitively, we expect this to be the case when the damping constant $\beta$ is positive. However the undamped case $\beta = 0$ is less clear. In this case there is no dissipation of energy, so is the rate of increase of energy quantifiable? Note that in this section we use the term ``weak convergence'' in the probabilistic sense, \textit{not} in the functional analytic sense.

We treat the undamped case first. Throughout this section we will use the interpretation of the solution $u: \Omega \times [0,\infty) \to \Hcal$ as an $\Hcal$-valued process. Recall the series representation of $u$,
\begin{equation*}
u = \sum_{k=1}^\infty Y^{b,\beta}_k \phi^b_k,
\end{equation*}
given in \eqref{seriesrep}.
\begin{thm}[$\beta = 0$]
Let $u$ be the solution to the SPDE \eqref{SWE} with $\beta = 0$.
\begin{enumerate}
\item If $b \neq N$, then $t^{-\frac{1}{2}}u(t)$ has a non-trivial weak limit in $\Hcal$ as $t \to \infty$.
\item If $b = N$, then $t^{-\frac{1}{2}}u(t)$ has no weak limit in $\Hcal$ as $t \to \infty$. However $u - Y^{N,\beta}_1 \phi^N_1$ and $Y^{N,\beta}_1 \phi^N_1$ are independent $\Hcal$-valued processes and $t^{-\frac{1}{2}}\left( u(t) - Y^{N,\beta}_1(t) \phi^N_1 \right)$ has a non-trivial weak limit in $\Hcal$ as $t \to \infty$.
\end{enumerate}
\end{thm}

\begin{proof}
Let $(\zeta_k)_{k=1}^\infty$ be an independent and identically distributed sequence of real standard Gaussian random variables. We start with (1), so that $\lambda^b_1 > 0$. For each $t \in [0,\infty)$ let
\begin{equation*}
\bar{u}(t) = \sum_{k=1}^\infty (2\lambda^b_k)^{-\frac{1}{2}} \left( t - (4\lambda^b_k)^{-\frac{1}{2}} \sin \left( (4\lambda^b_k)^\frac{1}{2} t \right) \right)^\frac{1}{2} \zeta_k \phi^b_k.
\end{equation*}
It can be easily checked that $\bar{u}(t)$ is a well-defined $\Hcal$-valued random variable with the same law as $u(t)$ for each $t \in [0,\infty)$. Now let
\begin{equation*}
u_\infty = \sum_{k=1}^\infty (2\lambda^b_k)^{-\frac{1}{2}} \zeta_k \phi^b_k,
\end{equation*}
so that $u_\infty$ is also a well-defined $\Hcal$-valued random variable. It is then simple to check using dominated convergence (see \cite[Proposition 2.5]{hambly2018}) that
\begin{equation*}
\lim_{t \to \infty}\Ebb\left[ \Vert t^{-\frac{1}{2}}\bar{u}(t) - u_\infty \Vert_\mu^2 \right] = 0,
\end{equation*}
so in particular $t^{-\frac{1}{2}}\bar{u}(t) \to u_\infty$ weakly as $t \to \infty$. Therefore the same weak convergence holds for $t^{-\frac{1}{2}}u(t)$.

We now tackle (2). The issue that forces us to consider this case separately is that $\lambda^N_1 = 0$, so the variance of $\langle t^{-\frac{1}{2}}u(t) , \phi^N_1 \rangle_\mu$ tends to infinity as $t \to \infty$. We deal with this by subtracting off the offending component, which is exactly $Y^{N,\beta}_1 \phi^N_1$. It is clearly independent of $u - Y^{N,\beta}_1 \phi^N_1$ by \eqref{seriesrep}. Now $\lambda^N_k > 0$ for all $k \geq 2$, so similar to (1) we let
\begin{equation*}
\bar{u}(t) = \sum_{k=2}^\infty (2\lambda^N_k)^{-\frac{1}{2}} \left( t - (4\lambda^N_k)^{-\frac{1}{2}} \sin \left( (4\lambda^N_k)^\frac{1}{2} t \right) \right)^\frac{1}{2} \zeta_k \phi^N_k,
\end{equation*}
which has the same law as $u(t) - Y^{N,\beta}_1(t) \phi^N_1$, and
\begin{equation*}
u_\infty = \sum_{k=2}^\infty (2\lambda^N_k)^{-\frac{1}{2}} \zeta_k \phi^N_k.
\end{equation*}
As with (1) we conclude that $t^{-\frac{1}{2}}\left( u(t) - Y^{N,\beta}_1(t) \phi^N_1 \right) \to u_\infty$ weakly as $t \to \infty$.
\end{proof}

We now tackle the damped case $\beta > 0$. It turns out that we must split this again into two subcases: $b \neq N$ and $b = N$.
\begin{thm}[$\beta > 0$]
Let $u$ be the solution to the SPDE \eqref{SWE} with $\beta > 0$. 
\begin{enumerate}
\item If $b \neq N$, then $u(t)$ has a non-trivial weak limit as $t \to \infty$.
\item If $b = N$, then $u(t)$ has no weak limit as $t \to \infty$. However $u - Y^{N,\beta}_1 \phi^N_1$ and $Y^{N,\beta}_1 \phi^N_1$ are independent $\Hcal$-valued processes, and $\left( u(t) - Y^{N,\beta}_1(t) \phi^N_1 \right)$ has a non-trivial weak limit as $t \to \infty$.
\end{enumerate}
\end{thm}

\begin{proof}
We do case (1) first. Observe that if $\beta > 0$ and $b \in 2^{F^0} \setminus \{ N \}$ then $V_\beta(\lambda, t)$ decays exponentially as $t \to \infty$ for any $\lambda \geq 0$. It follows that
\begin{equation}\label{Vintble}
\int_0^\infty V_\beta(\lambda,s)^2ds < \infty
\end{equation}
for all $\lambda > 0$, and so by It\={o}'s isometry we may define
\begin{equation*}
Z^{b,\beta}_k(t) = \int_0^t V_\beta(\lambda^b_k,s) \langle \phi^b_k, dW(s) \rangle_\mu
\end{equation*}
for each $t \in [0,\infty]$ and $k \geq 1$, which is an $\Hcal$-valued random variable. In the case $t \in [0,\infty)$ this evidently has the same law as $Y^{b,\beta}_k(t)$. Then for each $t \in [0,\infty)$ let
\begin{equation*}
\hat{u}(t) = \sum_{k=1}^\infty Z^{b,\beta}_k(t) \phi^b_k
\end{equation*}
and
\begin{equation*}
u_\infty = \sum_{k=1}^\infty Z^{b,\beta}_k(\infty) \phi^b_k.
\end{equation*}
It is clear that $\bar{u}(t)$ is an $\Hcal$-valued random variable with the same law as $u(t)$, for all $t \in [0,\infty)$. Now for any $t \in [0,\infty)$ we have by It\={o}'s isometry that
\begin{equation}\label{uhatestim}
\begin{split}
\Ebb\left[ \Vert \hat{u}(t) - u_\infty \Vert_\mu^2 \right] &= \sum_{k=1}^\infty \Ebb\left[ \left( Z^{b,\beta}_k(t) - Z^{b,\beta}_k(\infty) \right)^2 \right]\\
&= \sum_{k=1}^\infty \int_t^\infty V_\beta (\lambda^b_k,s)^2 ds\\
&= \sum_{k: \lambda^b_k \leq \beta^2} \int_t^\infty V_\beta (\lambda^b_k,s)^2 ds + \sum_{k: \lambda^b_k > \beta^2} \int_t^\infty V_\beta (\lambda^b_k,s)^2 ds.
\end{split}
\end{equation}
We treat each of these terms separately. As we mentioned in Proposition \ref{SWE2soln}, there are only finitely many $k$ such that $\lambda^b_k \leq \beta^2$, see \cite[Proposition 2.5]{hambly2018}. Then by \eqref{Vintble} we have that
\begin{equation*}
\begin{split}
\sum_{k: \lambda^b_k \leq \beta^2} \int_t^\infty V_\beta (\lambda^b_k,s)^2 ds &< \infty, \quad t \geq 0,\\
\lim_{t \to \infty}\sum_{k: \lambda^b_k \leq \beta^2} \int_t^\infty V_\beta (\lambda^b_k,s)^2 ds &= 0.
\end{split}
\end{equation*}
Now for the the $\{k: \lambda^b_k > \beta^2 \}$ sum we need to do some estimates. Our assumption that $\beta > 0$ allows us to improve on the estimates of Proposition \ref{SWE2soln}:
\begin{equation*}
\begin{split}
\sum_{k: \lambda^b_k > \beta^2} \int_t^\infty V_\beta (\lambda^b_k,s)^2 ds &= \sum_{k: \lambda^b_k > \beta^2} \frac{1}{\lambda^b_k - \beta^2} \int_t^\infty e^{-2\beta s} \sin^2\left( (\lambda^b_k - \beta^2)^\frac{1}{2} s \right) ds\\
&\leq \sum_{k: \lambda^b_k > \beta^2} \frac{1}{\lambda^b_k - \beta^2} \int_t^\infty e^{-2\beta s} ds\\
&= \frac{1}{2\beta} e^{-2\beta t} \sum_{k: \lambda^b_k > \beta^2} \frac{1}{\lambda^b_k - \beta^2}.
\end{split}
\end{equation*}
By \cite[Proposition 2.5]{hambly2018} the infinite sum above converges, so we have by dominated convergence that
\begin{equation*}
\begin{split}
\sum_{k: \lambda^b_k > \beta^2} \int_t^\infty V_\beta (\lambda^b_k,s)^2 ds &< \infty, \quad t \geq 0,\\
\lim_{t \to \infty}\sum_{k: \lambda^b_k > \beta^2} \int_t^\infty V_\beta (\lambda^b_k,s)^2 ds &= 0.
\end{split}
\end{equation*}
Setting $t = 0$ in \eqref{uhatestim}, we have now proven that
\begin{equation*}
\Ebb\left[ \Vert u_\infty \Vert_\mu^2 \right] < \infty,
\end{equation*}
and so $u_\infty$ is a well-defined $\Hcal$-valued random variable. From \eqref{uhatestim} we have also proven that
\begin{equation*}
\lim_{t \to \infty}\Ebb\left[ \Vert \hat{u}(t) - u_\infty \Vert_\mu^2 \right] = 0.
\end{equation*}
In particular this implies that $\hat{u}(t) \to u_\infty$ weakly as $t \to \infty$. Since $u(t)$ has the same law as $\hat{u}(t)$ for all $t$, this implies that $u(t) \to u_\infty$ weakly as $t \to \infty$.

In (2), we have the issue that $\lambda^N_1 = 0$ so $V_\beta(\lambda^N_1,\cdot)$ is not square-integrable, which precludes $u(t)$ from having a weak limit. We get around this issue by simply subtracting the associated term of the series representation of $u$, leaving only the square-integrable terms. We we still have $\lambda^N_k > 0$ for all $k \geq 2$, so by It\={o}'s isometry we may define
\begin{equation*}
Z^{N,\beta}_k(\infty) := \int_0^\infty V_\beta(\lambda^N_k,s) \phi^{N*}_kdW(s)
\end{equation*}
for $k \geq 2$. From the series representation \eqref{seriesrep} of $u$, observe that $Y^{N,\beta}_1(t) \phi^N_1$ is simply the component of $u(t)$ associated with the eigenfunction $\phi^N_1$, so that
\begin{equation*}
u(t) - Y^{N,\beta}_1(t) \phi^N_1 = \sum_{k=2}^\infty Y^{N,\beta}_k(t) \phi^N_k,
\end{equation*}
and the independence result is clear. For each $t$ we then define
\begin{equation*}
Z^{N,\beta}_k(t) = \int_0^t V_\beta(\lambda^N_k,s) \phi^{N*}_kdW(s)
\end{equation*}
and
\begin{equation*}
\hat{u}(t) = \sum_{k=2}^\infty Z^{N,\beta}_k(t) \phi^N_k,
\end{equation*}
so that $\hat{u}(t)$ has the same law as $u(t) - Y^{N,\beta}_1(t) \phi^N_1$. The proof proceeds from here in the same way as in the proof of (1)---we show that
\begin{equation*}
\Ebb\left[ \Vert u_\infty \Vert_\mu^2 \right] < \infty
\end{equation*}
and
\begin{equation*}
\lim_{t \to \infty}\Ebb\left[ \Vert \hat{u}(t) - u_\infty \Vert_\mu^2 \right] = 0
\end{equation*}
which imply the result.
\end{proof}

\bibliography{swebib2}
\bibliographystyle{alpha}

\end{document}